\documentclass[11pt, a4paper]{amsart}

\usepackage{tikz-cd}
\usepackage{amssymb}
\usepackage{mathrsfs}
\usepackage{amsmath}
\usepackage[shortlabels]{enumitem}
\usepackage[colorlinks,linkcolor=black,citecolor=black,urlcolor=black]{hyperref}
\usepackage[color = blue!20, bordercolor = black, textsize = tiny]{todonotes}
\usepackage[capitalise]{cleveref}
\usepackage{comment}

\numberwithin{equation}{section}
\newtheorem{thm}[subsection]{Theorem}
\newtheorem{cor}[subsection]{Corollary}
\newtheorem{lem}[subsection]{Lemma}
\newtheorem{prop}[subsection]{Proposition}
\newtheorem{thmalpha}{Theorem}

\theoremstyle{definition}
\newtheorem{df}[subsection]{Definition}
\newtheorem{rmk}[subsection]{Remark}
\newtheorem{exm}[subsection]{Example}
\newtheorem{const}[subsection]{Construction}

 \usepackage[ left=3cm, right=3cm, top = 2.8cm, bottom = 2.8cm ]{geometry}
\usepackage{stmaryrd}
\usepackage[normalem]{ulem}
\usepackage{bbm}

\makeatletter
\def\@tocline#1#2#3#4#5#6#7{\relax
  \ifnum #1>\c@tocdepth 
  \else
    \par \addpenalty\@secpenalty\addvspace{#2}%
    \begingroup \hyphenpenalty\@M
    \@ifempty{#4}{%
      \@tempdima\csname r@tocindent\number#1\endcsname\relax
    }{%
      \@tempdima#4\relax
    }%
    \parindent\z@ \leftskip#3\relax \advance\leftskip\@tempdima\relax
    \rightskip\@pnumwidth plus4em \parfillskip-\@pnumwidth
    #5\leavevmode\hskip-\@tempdima
      \ifcase #1
      \or\or \hskip 2em \or \hskip 2em \else \hskip 3em \fi%
      #6\nobreak\relax
    \dotfill\hbox to\@pnumwidth{\@tocpagenum{#7}}\par
    \nobreak
    \endgroup
  \fi}
\makeatother

\newcommand{\bE}{\mathbf{E}}

\newcommand{\F}{\mathbb{F}}

\renewcommand{\L}{\mathbb{L}}

\newcommand{\N}{\mathbb{N}}
\renewcommand{\P}{\mathbb{P}}
\newcommand{\Q}{\mathbb{Q}}

\newcommand{\Z}{\mathbb{Z}}

\newcommand{\cE}{\mathcal{E}}

\newcommand{\cM}{\mathcal{M}}

\newcommand{\cO}{\mathcal{O}}

\newcommand{\cY}{\mathcal{Y}}

\DeclareMathOperator{\Hom}{Hom}
\DeclareMathOperator{\Spec}{Spec}

\DeclareMathOperator{\Fil}{Fil}

\newcommand{\ul}{\underline}

\newcommand{\fib}{\mathrm{fib}}

\newcommand{\sat}{\mathrm{sat}}
\newcommand{\BGL}{\mathrm{BGL}}

\newcommand{\SmlSm}{\mathrm{SmlSm}}
\newcommand{\lSm}{\mathrm{lSm}}
\newcommand{\logSH}{\mathrm{logSH}}
\newcommand{\pt}{\mathrm{pt}}
\newcommand{\Set}{\mathrm{Set}}
\newcommand{\Fun}{\mathrm{Fun}}

\newcommand{\Sh}{\mathrm{Sh}}
\newcommand{\Sp}{\mathrm{Sp}}
\newcommand{\CAlg}{\mathrm{CAlg}}

\newcommand{\ch}{\mathrm{ch}}
\newcommand{\HK}{\mathrm{HK}}
\newcommand{\dR}{\mathrm{dR}}

\newcommand{\KH}{\mathrm{KH}}

\newcommand{\HC}{\mathrm{HC}}
\newcommand{\HP}{\mathrm{HP}}
\newcommand{\THH}{\mathrm{THH}}
\newcommand{\TC}{\mathrm{TC}}

\newcommand{\cts}{\mathrm{cts}}
\newcommand{\Pic}{\mathrm{Pic}}
\newcommand{\unit}{\mathbf{1}}

\newcommand{\et}{\mathrm{\acute{e}t}}
\newcommand{\setale}{\mathrm{s\acute{e}t}}

\newcommand{\Vect}{\mathrm{Vect}}

\newcommand{\sNis}{\mathrm{sNis}}

\def\lSm{\mathbf{lSm}}
\def\SmlSm{\mathbf{SmlSm}}

\newcounter{elno}

\usepackage[bbgreekl]{mathbbol}
\usepackage{relsize}

\DeclareSymbolFontAlphabet{\mathbb}{AMSb} 
\DeclareSymbolFontAlphabet{\mathbbl}{bbold} 
\newcommand{\cPrism}{\widehat{\mathlarger{\mathbbl{\Delta}}}}

\begin{document}
\title[The semistable $p$-adic deformation problem]{On the $p$-adic deformation problem \\ for the $K$-theory of semistable schemes}

\author{Federico Binda}
\address{Dipartimento di Matematica ``Federigo Enriques'',  Universit\`a degli Studi di Milano\\ Via Cesare Saldini 50, 20133 Milano, Italy}
\email[F. Binda]{federico.binda@unimi.it}

\author{Tommy Lundemo}
\address{Mathematisch Instituut, Universiteit Utrecht, Budapestlaan 6, 3584 CD Utrecht, Netherlands}
\email[T. Lundemo]{t.lundemo@uu.nl}

\author{Alberto Merici}
\address{Institut f\"ur Mathematik, Universit\"at Heidelberg\\ MATHEMATIKON, Im Neuenheimer Feld 205, 69120  Heidelberg, Germany}
\email[A. Merici]{merici@mathi.uni-heidelberg.de}

\author{Doosung Park}
\address{Department of Mathematics and Informatics, University of Wuppertal, Germany}
\email[D. Park]{dpark@uni-wuppertal.de}

\thanks{T.L.\ is funded by the NWO-grant VI.Veni.242.129; A.M.\ is funded by the Deutsche Forschungsgemeinschaft (DFG, German Research 
Foundation)
TRR 326 \textit{Geometry and Arithmetic of Uniformized Structures}, 
project number 444845124.
D.P.\ is funded by the DFG research training group GRK 2240: \emph{Algebro-Geometric Methods in Algebra, Arithmetic and Topology}.}

\date{\today}
\begin{abstract}
We establish a semistable generalization of the Beilinson–Bloch–Esnault–Kerz fiber square, relating the algebraic $K$-theory of a semistable scheme to its logarithmic topological cyclic homology. We prove that the obstruction to lifting $K$-theory classes is governed by the Hyodo-Kato Chern character. This answers the $p$-adic deformation problem for continuous $K$-theory in the semistable case, extending the work of Antieau–Mathew–Morrow–Nikolaus. As an application, we provide a purely K-theoretic proof of Yamashita's semistable $p$-adic Lefschetz $(1,1)$-theorem.
\end{abstract}
\maketitle
\setcounter{tocdepth}{1}
\section{Introduction}

Let $K$ be a complete discretely valued field of mixed characteristic $(0, p)$ with ring of integers ${\mathcal O}_K$ and perfect residue field $k$. If $X$ is smooth and proper over ${\mathcal O}_K$, the problem of lifting a class $x \in K_0(X_k ; {\Bbb Q})$ in the $K$-theory of the special fiber $X_k$ to the continuous $K$-theory $K^{\rm cts}_0(X ; {\Bbb Q})$ of $X$ is governed by the Hodge filtration on the de Rham cohomology of the generic fiber $X_K$ (see Section \ref{subsec:defk} for a more precise statement). This is a consequence of the cartesian square \begin{equation}\label{eq:introsquare1}\begin{tikzcd}K^{\rm cts}(X ; {\Bbb Q}) \ar{r} \ar{d} & K(X_k ; {\Bbb Q}) \ar{d} \\ \prod_{i \in {\Bbb Z}} {\rm Fil}^{\ge i} R\Gamma_{{\rm dR}}(X_K / K)[2i] \ar{r} & \prod_{i \in {\Bbb Z}} R\Gamma_{{\rm dR}}(X_K / K)[2i] \end{tikzcd}\end{equation} of Antieau--Mathew--Morrow--Nikolaus \cite{AMMN}, after Beilinson \cite{Bei2} and Bloch-Esnault-Kerz \cite{BEK14}. One core ingredient in its construction is \cite[Theorem A]{AMMN}, which provides a cartesian square \begin{equation}\label{eq:introsquare2}\begin{tikzcd}K(R ; {\Bbb Q}_p) \ar{r} \ar{d}{{\rm tr}} & K(R / p ; {\Bbb Q}_p) \ar{d}{{\rm tr}_{\rm crys}} \\ {\rm HC}^-(R ; {\Bbb Q}_p) \ar{r} & {\rm HP}(R ; {\Bbb Q}_p)\end{tikzcd}\end{equation} for any commutative ring $R$, Henselian along $(p)$, whose bottom horizontal map involves invariants built out of Hochschild homology.   

\subsection{The fiber square for log rings} To extend results of this form to the case of semistable special fiber $X_k$, one may want to make use of techniques from log geometry and its relaxed notion of smoothness. Moreover, as the relationship between the squares \eqref{eq:introsquare1} and \eqref{eq:introsquare2} comes to life from the Hochschild--Kostant--Rosenberg filtration, it is for this purpose natural to make use of Rognes' generalization of Hochschild homology to the logarithmic setting \cite{Rog}, since this theory admits an analogous filtration in light of \cite{BLPO}.

The role of affine schemes in log geometry is played by \emph{log rings} $(R, P)$. We have already taken a first step in generalizing the square \eqref{eq:introsquare2} to the logarithmic context, in that we established a cartesian square \cite[Theorem 1.5]{BLMP2} \begin{equation}\label{eq:introsquare3} \begin{tikzcd}{\rm TC}((R, P) ; {\Bbb Q}_p) \ar{r} \ar{d} & {\rm TC}((R / p, P) ; {\Bbb Q}_p) \ar{d} \\ {\rm HC}^-((R, P) ; {\Bbb Q}_p) \ar{r} & {\rm HP}((R, P) ; {\Bbb Q}_p),\end{tikzcd}\end{equation} where all terms involved are built out of Rognes' log (topological) Hochschild homology. 

Under additional geometric assumptions, we can ``merge'' \eqref{eq:introsquare2} and  \eqref{eq:introsquare3} to get the following result, relating ordinary $K$-theory to logarithmic invariants: 

\begin{thmalpha}[See Theorem \ref{thm:beilfibernologktheory}]\label{thm:logbek}  Let $(R, P)$ be an ${\mathcal O}_K^\sharp$-algebra which is vertical, regular, and log regular. If $R$ is Henselian along $(p)$, then there exists a natural cartesian square
\[
\begin{tikzcd}
K(R;\Q_p)\ar[d]\ar[r]&
K(R/p;\Q_p)\ar[d]
\\
\HC^-((R,P)/\cO_K^\sharp;\Q_p)\ar[r]&
\HP((R,P)/\cO_K^\sharp;\Q_p)
\end{tikzcd}
\]
of spectra.
\end{thmalpha}
Here, ${\mathcal O}_K^\sharp$ denotes ${\mathcal O}_K$ equipped with its canonical log structure ${\mathcal O}_K - \{0\}$ (a chart is given by $\N\to \mathcal{O}_K, 1\mapsto \pi$ for $\pi\in \mathcal{O}_K$ a uniformizer). We refer to the beginning of Section \ref{sec:logbek} for an explanation of the adjectives involved in the statement. It is important, however, to remark now that typical examples of regular, log regular, and vertical $\mathcal{O}_K^{\sharp}$-algebras are the rings of the form $\mathcal{O}_K[X_1, \ldots, X_n]/(X_1\cdots X_i-\pi)$, for $1\leq i\leq n$, equipped with its canonical log structure. This is precisely the local structure of (strict) semistable families over $\Spec(\mathcal{O}_K)$. The proof is presented in Section \ref{sec:logbek}. To achieve this, we analyze the interaction between vertical log structures and rationalized invariants, crucially using the residue sequences of Rognes--Sagave--Schlichtkrull, see \cite{RSS15}. This allows us to reduce the problem to the case of trivial log structures.

\subsection{$p$-adic deformation of \texorpdfstring{$K$}{K}-theory classes and Hyodo-Kato cohomology}\label{subsec:defk} Let us now make precise the claim that the lifting problem considered at the beginning of this introduction is governed by the Hodge filtration of $R\Gamma_{{\rm dR}}(X_K / K)$. In good reduction, there is a crystalline Chern character \[{\rm ch}_{{\rm crys}} \colon K_0(X_k) \to \bigoplus_{i \ge 0} H^{2i}_{{\rm crys}}(X_k ; {\Bbb Q}_p)\] landing in crystalline cohomology. By the de Rham-to-crystalline comparison, this determines a map to $\bigoplus_{i \ge 0} H_{{\rm dR}}^{2i} (X_K / K)$. On $\pi_0$, this recovers the right-hand vertical map of \eqref{eq:introsquare1}, and \cite[Theorem D]{AMMN} states that a class $x \in K_0(X_k ; {\Bbb Q})$ lifts to $K_0^{{\rm cts}}(X ; {\Bbb Q})$ if and only if its image in $\bigoplus_{i \ge 0} H^{2i}_{{\rm dR}}(X_K / K)$ comes from $\bigoplus_{i \ge 0} {\rm Fil}^{\ge i} H^{2i}_{{\rm dR}}(X_K / K)$. This is in fact a generalization of the original statement proved by \cite{BEK14}. 

In bad reduction it is possible to govern the obstruction to lifting classes by means of derived de Rham cohomology, as explained in \cite{AMMN} following Beilinson \cite{Bei2}. More precisely, when $X$ is not smooth nor proper over $\Spec(\cO_K)$, a refined obstruction to lifting classes has been constructed by Beilinson and Antieau--Mathew--Morrow--Nikolaus: for $x \in K_i({X}_k;\mathbb{{Q}})$, there exists a class \[o(x)\in \bigoplus_{r\geq 0} H^{2r-i}(L\Omega_{X} / L\Omega^{\geq r}_{{X}})_{\mathbb{Q}_p}\] such that $x$ lifts to $K^{\rm cts}_i({X};\mathbb{Q}_p)$ if and only if $o(x)=0$. Here, $L\Omega_{X}$ denotes the $p$-adic derived de Rham cohomology of ${X}$. See \cite[Theorem E]{AMMN}.

However, the obstruction class $o(x)$ is rather implicit and lacks a crystalline interpretation, in contrast with  the smooth case.

To tackle this problem, we recall the following fact: when $X$ is proper and has semistable reduction over $\mathcal{O}_K$, the de Rham cohomology of the  generic fiber of ${X}$ can be recovered by means of the Hyodo--Kato cohomology
\begin{equation}\label{eq:HK}
    R\Gamma_{\dR}({X}_K) \simeq R\Gamma_{\HK}({X}_k/k^0)\otimes_{W(k)} K  \end{equation}
of the \emph{logarithmic} special fiber, seen as a smooth log scheme over the \emph{logarithmic point} $k^0 = ( \mathbb{N}\to k, 1\mapsto 0)$. The Hyodo--Kato cohomology is a $(\varphi, N)$-module over $W(k)$, and can be defined using an appropriate version of the crystalline site following work of Kato and Fontaine--Illusie, or in terms of log de Rham-Witt complexes (note however that functoriality is tricky: see \cite{BGV} for a motivic approach working directly on the rigid analytic generic fiber of $\mathfrak{X} = X^{\wedge}_p$). 

In 
Section \ref{sec:hkchern}, we consider a \emph{Hyodo--Kato Chern character} ${\rm ch}_{{\rm HK}}$ from the (ordinary, rationalized) $K$-theory of $X_k$, and express it in terms of a trace map from $K$-theory, much in analogy with the study of the crystalline Chern character ${\rm ch}_{{\rm crys}}$ in \cite{AMMN}. Combining the resulting properties with Theorem \ref{thm:logbek}, we arrive at the following analog of \cite[Theorem D]{AMMN}: 

\begin{thmalpha}[See Theorem \ref{thm:hkchern2}]\label{thm:hkchern} Let $X$ be a proper flat scheme over $\cO_K$ with semistable special fiber $X_k$.
Then a class $x\in K_0(X_k;\Q)$ lifts to $K_0^\cts(X;\Q)$ if and only if the Hyodo-Kato Chern character $\ch_\HK(x)$ is contained in $\bigoplus_{i= 0}^\infty \Fil^{\geq i}H_\dR^{2i}(X_K;\Q_p)$ under the Hyodo-Kato equivalence.
\end{thmalpha}

In fact, we have a general version involving higher $K$-groups, analogous to \eqref{eq:introsquare1}, see Theorem \ref{rig.4}. 

To the best of the authors' knowledge, Theorem \ref{thm:hkchern} represents the first $K$-theoretic application of logarithmic trace invariants of genuinely geometric origin. We anticipate progress on \cite[Conjecture 1.5]{Lun}, which will lead to a more flexible and transparent relationship between algebraic $K$-theory and logarithmic trace invariants. The fact that results like Theorem \ref{thm:hkchern} are already within reach with the present constructions suggests that this is a worthwhile endeavor. 

To deduce Theorem \ref{thm:hkchern},  we follow the blueprint of \cite{AMMN}. In Section 3, we construct the Hyodo-Kato Chern character, using the theory   of Chern orientations in logarithmic motivic homotopy theory in the sense of \cite{BPO2} and apply it to the logarithmic de Rham-Witt complex $W\Omega_{X/(k,P)}$. This yields a geometric construction of Chern classes in Hyodo-Kato cohomology, compatible with the classical crystalline Chern character. In Section 4, we focus on the rationalized filtrations of logarithmic cyclic homology. We use the logarithmic Hochschild-Kostant-Rosenberg theorem \cite{BLPO} to split these filtrations in terms of the Hodge filtration on the de Rham cohomology of the generic fiber, by means of the Hyodo-Kato isomorphism. This leads to a $K$-theoretic fiber square for semistable schemes involving de Rham cohomology (Theorem \ref{rig.4}). In Section 5, we compare the trace-theoretic construction from Section 2 with the geometric Chern character from Section 3, and prove that, up to a scalar factor, the trace map identifies with the Hyodo-Kato Chern character under the Hyodo-Kato equivalence.

\subsection{A semistable $p$-adic Lefschetz (1,1)-Theorem}
Note that the $K_0$-groups appearing in the statement of Theorem \ref{thm:hkchern} are classical, i.e., non-logarithmic. For line bundles, i.e., for $\Pic(X;\Q)$, the obstruction considered in Theorem \ref{thm:hkchern} is in fact well-known: it is the content of  Yamashita's semistable $p$-adic Lefschetz (1,1)-Theorem  \cite[Theorem 0.1]{Yam11} (in turn, a generalization of a theorem of Berthelot-Ogus \cite[Theorem 3.8]{BO}). 

In fact, Yamashita considers also a logarithmic variant of this result, involving the \emph{logarithmic Picard group}, i.e., $H^1(X, \mathcal{M}^{\rm gp})$, where $\mathcal{M}$ is the sheaf of monoids describing the log structure on $X$ (see \cite{Kato-dieudonne} and \cite[Section 3.3]{MR1307397}). In order to recover (and generalize) his result, we need to be able to replace $K$-theory with some version of logarithmic $K$-theory, computable with effective trace methods.

Since its inception, a severe limitation of logarithmic $\rm TC$ has been the lack of a meaningful trace map from algebraic $K$-theory.
To bridge the gap, we look to recent developments in logarithmic \textit{homotopy invariant} $K$-theory, as discussed in \cite{Par}:
for every finite-dimensional Noetherian ${\mathcal O}_K^{\sharp}$-algebra $(R,P)$ which is regular, log regular, and vertical, by \cite[Theorem F]{Par2} we have a log cyclotomic trace $ K(R[1/p])\to \TC(R,P)$ fitting in the following commutative diagram of spectra:
\begin{equation}\label{eq:log-KH-trc}
\begin{tikzcd}[row sep=small]
\KH((R,P);\Q_p)\ar[d,"\simeq"']\ar[r,leftarrow]&
\KH(R;\Q_p)\ar[d,"\simeq"]\ar[r]&
\KH(R/p;\Q_p)\ar[d,"\simeq"]
\\
K(R[1/p];\Q_p)\ar[d,"{\rm trc}"']\ar[r,leftarrow]&
K(R;\Q_p)\ar[d,"{\rm trc}"]\ar[r]&
K(R/p;\Q_p)\ar[d,"{\rm trc}"]
\\
\TC((R,P);\Q_p)\ar[r,leftarrow]&
\TC(R;\Q_p)\ar[r]&
\TC(R/p;\Q_p),
\end{tikzcd}
\end{equation}
where the left (resp.\ right) upper vertical equivalence is due to \cite[Example 3.6, Theorem 3.7]{Par3}, (resp.\ Proposition \ref{square.7}).
The use of homotopy $K$-theory in light of the $p$-adic deformation problem is in fact a rather natural thing to do, thanks to the following result:

\begin{thmalpha}[See Theorem \ref{prop:old-thm-A}]\label{thm:logbekkh} Let $(R, P)$ be a finite-dimensional Noetherian ${\mathcal O}_K^{\sharp}$-algebra which is regular, log regular, and vertical. If $R$ is Henselian along $(p)$, then there exists a natural cartesian square \[\begin{tikzcd}K(R[1/p] ; {\Q_p}) \ar{r} \ar{d} & {\rm KH}((R / p, P) ; {\Q_p}) \ar{d} \\ {\rm HC}^-((R, P) / {\mathcal O}_K^\sharp ; {\Bbb Q}_p) \ar{r} & {\rm HP}((R, P) / {\mathcal O}_K^\sharp ; {\Bbb Q}_p)\end{tikzcd}\] of spectra, where the top horizontal map is induced by the composition \[
K(R[1/p])\simeq \KH(R[1/p])\simeq \KH(R,P)\to \KH(R/p,P),
\] 
and the vertical maps are given by the traces.
\end{thmalpha}
Using Theorem \ref{thm:logbekkh}, we can construct a meaningful variant of Theorem \ref{thm:hkchern} involving log $K$-theory of the special fiber $X_k$. See Theorem \ref{rig.5}. As a matter of fact, we can also recover Yamashita's theorem, which follows directly from the general $K$-theoretic statement. See Section \ref{sect:Yamashita}.

\subsection{A characteristic zero story}
Finally, as a complement to our main focus on the $p$-adic deformation problem, we conclude in Section \ref{sec.char0} with a discussion of the characteristic zero setting.  Assuming $K$ is an algebraic extension of $\mathbb{Q}$, we consider proper semistable schemes over $K[[t]]$ and establish a semistable generalization of a theorem of Morrow \cite[Theorem 1.1]{Mor14}. While Morrow characterizes the liftability of $K$-theory classes for smooth schemes using the flat filtration on standard de Rham cohomology, the presence of singularities in the semistable case necessitates a logarithmic approach. We prove that a class in the (homotopy) $K$-theory of the special fiber lifts to the infinitesimal neighborhood if and only if its trace lifts appropriately within the relative logarithmic de Rham cohomology. We view this result primarily as a supplementary application of our methods rather than a definitive parallel to Theorem \ref{thm:hkchern}. In the $p$-adic case, we identify the obstruction with the Hyodo-Kato Chern character relative to the logarithmic point $(k, \mathbb{N})$; however, in the characteristic zero context, we work relative to $K$ with the trivial log structure.

\subsection{Notation} Throughout, we assume some familiarity with the basics of log geometry; we use \cite{Ogu} as our main reference. We will also freely use variants of logarithmic ${\rm THH}$ arising from \cite{Rog}, and we refer to \cite[Sections 2 and 3]{BLPO2} for a streamlined overview.
Unless otherwise stated,
every (formal) scheme is assumed to be quasi-compact and quasi-separated,
and every formal scheme is a $p$-adic formal scheme. Most of the time, we will consider log schemes which are fine and saturated (fs for short). For a log scheme $X$, we usually write $\underline{X}$ to denote the underlying scheme. Recall that for a topology $\tau$ on the category of (formal) schemes,
the \emph{strict $\tau$-topology} on the category of fs log schemes is defined to be the topology generated by the families $\{f_i\colon X_i\to X\}_{i\in I}$ such that each $f_i$ is strict and $\{\ul{f_i}\colon \ul{X_i}\to \ul{X}\}_{i\in I}$ is a $\tau$-covering. More generally, see \cite[Section 3.1]{BPO} for a recollection of different relevant topologies on fs log schemes. 

\subsection{Acknowledgments} We wish to thank Ben Antieau for several interesting conversations around Theorem B, and Hélène Esnault for providing comments on an earlier version of this manuscript. We are also very grateful to  Georg Tamme for pointing out a gap in an earlier version of Proposition \ref{square.7}. 

\section{A logarithmic Bloch--Esnault--Kerz square}\label{sec:logbek}  We begin with some technical recollections:

\subsection{Vertical log structures} Let $\varphi \colon P \to Q$ be a map of commutative monoids. We say that $\varphi$ is \emph{vertical} \cite[Definition I.4.3.1(4)]{Ogu} if its cokernel $(Q / P)^{\rm int}$ in integral monoids is a group.
\begin{df}
    A morphism of fs log schemes $f\colon X\to S$ is called \emph{vertical} if the induced map $\cM_{S,f(x)}\to \cM_{X,x}$ is vertical for every point $x\in X$.
By \cite[Proposition I.4.3.3(6)]{Ogu}, if $f$ has a chart $P\to Q$ that is vertical, then $f$ is vertical too.
\end{df}
We shall use the following observation: If $({\Bbb Z}_p, \langle p \rangle) \to (R, P)$ is vertical, then the rationalized log scheme ${\rm Spec}(R \otimes_{{\Bbb Z}_p} {\Bbb Q}_p, P)$ carries the trivial log structure (otherwise, the resulting map to ${\rm Spec}({\Bbb Q}_p, \langle p \rangle) = {\rm Spec}({\Bbb Q}_p)$ would fail to be vertical).

\begin{exm}\label{exm:vertical} The morphism ${\Bbb N} \to {\Bbb N} \oplus {\Bbb N}, n \mapsto (n, n)$ is vertical, while the morphism ${\Bbb N} \to {\Bbb N} \oplus {\Bbb N},  n \mapsto (n, 0)$ is not. A classical example of a vertical morphism is given by the structure map $X\to \Spec(\Z_p, \langle p \rangle)$, where $X$ is a regular $\Z_p$ scheme with compactifying log structure given by the inclusion $X[1/p] \subset X$. 
\end{exm}

\subsection{Removing vertical log structures} We now take the first steps toward formulating variants of the fiber square \eqref{eq:introsquare3} that involve $K$-theory. Let us first observe that rationalized log ${\rm HC}$ vanishes on ${\mathbb F}_p$-algebras:

\begin{lem}\label{lem:hcq0} If $(R, P)$ is an ${\Bbb F}_p$-algebra, then ${\rm HC}((R, P) ; {\Bbb Q}_p) \simeq 0$.  
\end{lem}

\begin{proof} By assumption, $R / p \cong R$. Thus the upper horizontal map in the (co)cartesian square \eqref{eq:introsquare3} is an equivalence, and hence so is its lower horizontal map. 
\end{proof}

There is a notion of \emph{log regularity} introduced by Kato \cite{zbMATH00706681}; see \cite[Theorem III.1.11.1]{Ogu} for useful characterizations. 

\begin{prop}\label{prop:tccartnolog} Let $(R, P)$ be a regular log regular $\Z_p$-algebra.
Assume that $\Spec(R[1/p],P)$ has the trivial log structure.
Then the induced square of spectra
\[
\begin{tikzcd}
\TC(R;\Q_p)\ar[d]\ar[r]&
\TC(R/p;\Q_p)\ar[d]
\\
\TC((R,P);\Q_p)\ar[r]&
\TC((R/p,P);\Q_p)
\end{tikzcd}
\]
is cartesian.
\end{prop}

\begin{proof} Consider the commutative cube \[\begin{tikzcd}[row sep = tiny, column sep=small]
&
{\rm HC}^-(R ; {\Bbb Q}_p)
\ar{rr}{}
\ar[]{dd}[near end]{}
& & {\rm HP}(R ; {\Bbb Q}_p)
\ar{dd}{}
\\
{\rm TC}(R ; {\Bbb Q}_p)
\ar[crossing over]{rr}[near start]{}
\ar{dd}[swap]{}
\ar{ur}
& & {\rm TC}(R / p ; {\Bbb Q}_p)
\ar{ur}
\\
&
{\rm HC}^-((R, P) ; {\Bbb Q}_p)
\ar[near start]{rr}{}
& & {\rm HP}((R, P) ; {\Bbb Q}_p)
\\
{\rm TC}((R, P) ; {\Bbb Q}_p)
\ar{ur}
\ar{rr}
& & {\rm TC}((R / p, P) ; {\Bbb Q}_p).
\ar[crossing over, leftarrow, near start]{uu}{}
\ar{ur}
\end{tikzcd}\]  The top face is cartesian by Antieau--Mathew--Morrow--Nikolaus \cite[Corollary 3.9]{AMMN}, while the bottom face is the cartesian square \eqref{eq:introsquare3}. It thus suffices to prove that the back face is cartesian: We will show that the canonical map \begin{equation}\label{eq:hctologhc}{\rm HC}(R ; {\Bbb Q}_p) \to {\rm HC}((R, P) ; {\Bbb Q}_p)\end{equation} is an equivalence. We work Zariski locally on $R$. Since $R$ is regular and $(R,P)$ is log regular, we may assume that $P=\N^d$ for some integer $d\geq 0$ \cite[Theorem III.1.11.6]{Ogu}. We proceed to prove that \eqref{eq:hctologhc} is an equivalence by induction on $d$, the case $d = 0$ being clear. We now apply the residue sequence of Rognes--Sagave--Schlichtkrull \cite[Theorem 5.5, Example 5.7]{RSS15} to infer a fiber sequence of spectra \[
\HC((R/\alpha(x),\N^{d-1});\Q_p)
\to
\HC((R,\N^{d-1});\Q_p)
\to
\HC((R,\N^d);\Q_p),
\]
where $\alpha\colon \N^d\to R$ denotes the log structure map and $x = (1, 0, \dots, 0)$. Since $({\Bbb Z}_p, \langle p \rangle) \to (R, P)$ is vertical, $R / \alpha(x)$ is an ${\Bbb F}_p$-algebra (cf.\ Example \ref{exm:vertical}). Thus Lemma \ref{lem:hcq0} applies to conclude that the first term of the fiber sequence vanishes, so that its second map is an equivalence. However, the map \eqref{eq:hctologhc} is an equivalence for $P = {\Bbb N}^{d - 1}$ by induction hypothesis, concluding the proof. 
\end{proof}

\subsection{From absolute to relative}
Suppose now that $K$ is a complete discrete valuation field of mixed characteristic $(0, p)$ with ring of integers ${\mathcal O}_K$ and perfect residue field $k$. We shall write ${\mathcal O}_K^\sharp$ for the log ring $({\mathcal O}_K, {\mathcal O}_K  - \{0\})$; this is the log ring associated to the pre-log ring $({\mathcal O}_K, \langle \pi \rangle)$ for any chosen uniformizer $\pi$. 

\begin{prop}\label{prop:nodvrbase} For any ${\mathcal O}_K^\sharp$-algebra $(R, P)$, the canonical map \[{\rm HC}((R, P) ; {\Bbb Q}_p) \xrightarrow{} {\rm HC}((R, P) / {\mathcal O}_K^\sharp ; {\Bbb Q}_p)\] is an equivalence.
\end{prop}

\begin{proof} It is observed in \cite[Proof of Proposition 4.10]{AMMN} that ${\Bbb L}_{{\mathcal O}_K / W(k)}$ is quasi-isogenous to $0$. The assumption that $k$ is perfect gives ${\Bbb L}_{k / {\Bbb F}_p} \simeq 0$, which in turn implies that $({\Bbb L}_{W(k) / {\Bbb Z}})^\wedge_p \simeq 0$ (as this can be checked after reducing mod $p$). 
By \cite[(3.5)]{BLPO},
we have an equivalence
\[
\fib(\cO_K\xrightarrow{\pi} \cO_K)
\simeq
\fib(\L_{\cO_K/\Z}\to \L_{\cO_K^\sharp/\Z}).
\]
Combining these observations with the transitivity sequence \[{\mathcal O}_K \otimes_{W(k)} {\Bbb L}_{W(k) / {\Bbb Z}} \xrightarrow{} {\Bbb L}_{{\mathcal O}_K / {\Bbb Z}} \xrightarrow{} {\Bbb L}_{{\mathcal O}_K / W(k)},\] we conclude that $({\Bbb L}_{{\mathcal O}_K^\sharp / {\Bbb Z}})^\wedge_p \otimes_{{\Bbb Z}_p} {\Bbb Q}_p \simeq 0$. Hence the transitivity sequence \[{\Bbb L}_{{\mathcal O}_K^\sharp / {\Bbb Z}} \xrightarrow{} {\Bbb L}_{(R, P) / {\Bbb Z}} \xrightarrow{} {\Bbb L}_{(R, P) / {\mathcal O}_K^\sharp}\] implies that $({\Bbb L}_{(R, P) / {\Bbb Z}})^\wedge_p \otimes_{{\Bbb Z}_p} {\Bbb Q}_p \xrightarrow{} ({\Bbb L}_{(R, P) / {\mathcal O}_K^\sharp})^\wedge_p \otimes_{{\Bbb Z}_p} {\Bbb Q}_p$ is an equivalence, which implies that ${\rm HH}((R, P) ; {\Bbb Q}_p) \to {\rm HH}((R, P) / {\mathcal O}_K^\sharp ; {\Bbb Q}_p)$ by the log ${\rm HKR}$-filtration \cite[Theorem 1.1]{BLPO}. Commuting the $S^1$-coinvariants with $p$-completions as in \cite[Proof of Proposition 4.10]{AMMN}, we conclude. 
\end{proof}

\subsection{Involving \texorpdfstring{$K$}{K}-theory} Finally, we arrive at a logarithmic variant of the square \eqref{eq:introsquare2} which involves $K$-theory: 

\begin{thm}\label{thm:beilfibernologktheory} Let $(R, P)$ be an ${\mathcal O}_K^\sharp$-algebra which is vertical, regular, and log regular. If $R$ is Henselian along $(p)$, then there exists a natural cartesian square \begin{equation}
\label{square.4.1}
\begin{tikzcd}
K(R;\Q_p)\ar[d]\ar[r]&
K(R/p;\Q_p)\ar[d]
\\
\HC^-((R,P)/\cO_K^\sharp;\Q_p)\ar[r]&
\HP((R,P)/\cO_K^\sharp;\Q_p)
\end{tikzcd}
\end{equation}
of spectra.
\end{thm}

\begin{proof} We claim that each square in the diagram \begin{equation}
\label{4.1}
\begin{tikzcd}[row sep = small]
K(R;\Q_p)\ar[d]\ar[r]&
K(R/p;\Q_p)\ar[d]
\\
\TC(R;\Q_p)\ar[d]\ar[r]&
\TC(R/p;\Q_p)\ar[d]
\\
\TC((R,P);\Q_p)\ar[r]\ar[d]&
\TC((R/p,P);\Q_p)\ar[d]&
\\
\HC^-((R,P);\Q_p)\ar[r]\ar[d]&
\HP((R,P);\Q_p)\ar[d]
\\
\HC^-((R,P)/\cO_K^\sharp;\Q_p)\ar[r]&
\HP((R,P)/\cO_K^\sharp;\Q_p)
\end{tikzcd}
\end{equation}
is cartesian.
For the top square, this is \cite[Theorem 4.36]{CMM}.
For the second square from the top, this is Proposition \ref{prop:tccartnolog}. The third square from the top is \eqref{eq:introsquare3}, while the bottom square is a consequence of Proposition \ref{prop:nodvrbase}. 
\end{proof}

For any scheme $X$ over $\cO_K$,
we will use the continuous K-theory (after Beilinson \cite{Bei2})
\[
K^\cts(X)
:=
\lim_mK(X/p^m),
\;
K^\cts(X;\Z_p)
:=
\lim_m K(X/p^m;\Z_p).
\]
We set $K^\cts(X;\Q):=K^\cts(X)\otimes \Q$ and $K^\cts(X;\Q_p):=K^\cts(X;\Z_p)\otimes \Q$.
We will use similar notations for any formal scheme over $\cO_K$.

\begin{cor}
\label{cor:beilfibernologktheory}
Let $\cY$ be a vertical regular log regular formal log scheme over $\cO_K^\sharp$.
Then there exists a natural cartesian square
\[
\begin{tikzcd}
K^\cts(\ul{\cY};\Q)\ar[d]\ar[r]&
K(\ul{\cY}_k;\Q)\ar[d]
\\
\HC^-(\cY/\cO_K^\sharp;\Q_p)\ar[r]&
\HP(\cY/\cO_K^\sharp;\Q_p)
\end{tikzcd}
\]
of spectra,
where $\cY_k$ denotes the special fiber of $\cY$.
\end{cor}
\begin{proof}
The $\Q_p$-coefficient version is a direct consequence of Theorem \ref{thm:beilfibernologktheory}.
To obtain the $\Q$-coefficient version,
note that the induced square of spectra
\[
\begin{tikzcd}
K^\cts(\ul{\cY};\Q)\ar[d]\ar[r]&
K(\ul{\cY}_k;\Q)\ar[d]
\\
K^\cts(\ul{\cY};\Q_p)\ar[r]&
K(\ul{\cY}_k;\Q_p)
\end{tikzcd}
\]
is cartesian by \cite[Corollary 3.8]{AMMN} (or by Geisser--Hesselholt \cite{GH}).
\end{proof}

While Theorem \ref{thm:beilfibernologktheory} gives a logarithmic variant of the fiber square \eqref{eq:introsquare2}, the $K$-theory terms do not yet interact with the log structure. To overcome this, we shall replace $K$-theory with ${\rm KH}$, and use log cyclotomic traces. The passage from $K$ to $\KH$ in this context is rather harmless, in light of the following:
 
\begin{prop}
\label{square.7}
Let $A$ be a Noetherian $\F_p$-algebra,
and let $\mathfrak{p}_1,\ldots,\mathfrak{p}_m$ be the minimal prime ideals of $A$.
Assume that for every nonempty subset $I:=\{i_1,\ldots,i_r\}$ of $\{1,\ldots,m\}$,
$A_I:=A_{i_1}\otimes_A \cdots \otimes_A A_{i_r}$ is regular,
where $A_i:=A/\mathfrak{p}_i$ for each $i$.
Then the natural map of spectra
\[
K(A;\Q_p)
\to
\KH(A;\Q_p)
\]
is an equivalence.
\end{prop}


\begin{proof}
By assumption, the rings $A_i$ form  a  cdh cover, and $A\to A_\bullet$ is the associate cube. 
By \cite[Corollary A.4]{LT},
we have 
a cartesian square of spectra
\[
\begin{tikzcd}
K(A)\ar[d]\ar[r]&
\TC(A)\ar[d]
\\
\lim_{I\subset \{1,\ldots,m\},I\neq \emptyset}
K(A_I)\ar[r]&
\lim_{I\subset \{1,\ldots,m\},I\neq \emptyset}
\TC(A_I).
\end{tikzcd}
\]
Noting that the limit is finite, this yields a cartesian square of spectra
\[
\begin{tikzcd}
K(A;\Q_p)\ar[d]\ar[r]&
\TC(A;\Q_p)\ar[d]
\\
\lim_{I\subset \{1,\ldots,m\},I\neq \emptyset}
K(A_I;\Q_p)\ar[r]&
\lim_{I\subset \{1,\ldots,m\},I\neq \emptyset}
\TC(A_I;\Q_p).
\end{tikzcd}
\]
We have $K(A_I;\Q_p)\simeq \KH(A_I;\Q_p)$ for each $I$, and $\KH$ satisfies cdh descent.
Hence it suffices to show
\[
\TC(A;\Q_p)\simeq \lim_{I\subset \{1,\ldots,m\},I\neq \emptyset}
\TC(A_I;\Q_p).
\]

Since $\THH(A)$ is an $A$-algebra,
the rationalizations of $\THH(A)$ and $\THH(A;\Z_p)$ are zero.
Together with \cite[(3.8)]{EM23},
we have $\TC(A)\simeq \TC(A;\Z_p)$.
Hence we have $\TC(A;\Q)\simeq \TC(A;\Q_p)$.
Similarly, $\TC(A_I;\Q)\simeq \TC(A_I;\Q_p)$ for each $I$.
To conclude,
note that $\TC(-;\Q)$ is a cdh sheaf of spectra on the opposite category of $\F_p$-algebras by \cite[Proof of Corollary 4.20]{EM23}.
\end{proof}


For a finite-dimensional Noetherian fs log scheme $X$,
the homotopy K-theory $\KH(X)$ is defined in \cite[Definition 4.4.2]{Par}.
This extends Weibel's homotopy K-theory:
By \cite[Example 3.6, Theorem 3.7]{Par3},
if $X$ is regular and log regular,
then we have a natural equivalence
$
\KH(X)
\simeq
\KH(X-\partial X).
$ 

We can now complete the proof of Theorem \ref{thm:logbekkh}, which we recall:
\begin{thm}\label{prop:old-thm-A}
 The composition \[
K(R;\Q_p)\to K(R/p;\Q_p)\to \TC(R/p;\Q_p)\to \TC((R/p,P);\Q_p)\] 
factors through the log cyclotomic trace $\KH((R/p,P);\Q_p)\to \TC((R/p,P);\Q_p)$, inducing a cartesian square of spectra:
    \[\begin{tikzcd}K(R[1/p] ; {\Bbb Q}_p) \ar{r} \ar{d} & {\rm KH}((R / p, P) ; {\Bbb Q}_p) \ar{d} \\ {\rm HC}^-((R, P) / {\mathcal O}_K^\sharp ; {\Bbb Q}_p) \ar{r} & {\rm HP}((R, P) / {\mathcal O}_K^\sharp ; {\Bbb Q}_p).\end{tikzcd}\]
\end{thm}

\begin{proof}
Since $K(R[1/p] ; {\Bbb Q}_p) \simeq {\rm KH}((R, P) ; {\Bbb Q}_p)$ and the bottom square of \[\begin{tikzcd}[row sep = small]{\rm KH}((R, P) ; {\Bbb Q}_p) \ar{r} \ar{d} & {\rm KH}((R/p, P) ; {\Bbb Q}_p) \ar{d} \\ {\rm TC}((R, P) ; {\Bbb Q}_p) \ar{d} \ar{r} & {\rm TC}((R/p, P) ; {\Bbb Q}_p) \ar{d} \\ {\rm HC}^-((R, P) / {\mathcal O}_K^\sharp ; {\Bbb Q}_p) \ar{r} & {\rm HP}((R, P) / {\mathcal O}_K^\sharp ; {\Bbb Q}_p)\end{tikzcd}\] is cartesian
 by combining  \eqref{eq:introsquare3} with Proposition \ref{prop:nodvrbase}, it suffices to construct the top square and show that it is cartesian. This appears as the back face of the cube \[\begin{tikzcd}[row sep = tiny, column sep=small]
&
{\rm KH}((R, P) ; {\Bbb Q}_p)
\ar{rr}{}
\ar[]{dd}[near end]{}
& & {\rm KH}((R / p , P) ; {\Bbb Q}_p)
\ar[dd,dotted]
\\
{\rm KH}(R ; {\Bbb Q}_p)
\ar[crossing over]{rr}[near start]{}
\ar{dd}[swap]{}
\ar{ur}
& & {\rm KH}(R / p ; {\Bbb Q}_p)
\ar{ur}
\\
&
{\rm TC}((R, P) ; {\Bbb Q}_p)
\ar[near start]{rr}{}
& & {\rm TC}((R/p, P) ; {\Bbb Q}_p)
\\
{\rm TC}(R ; {\Bbb Q}_p)
\ar{ur}
\ar{rr}
& & {\rm TC}(R / p ; {\Bbb Q}_p)
\ar[crossing over, leftarrow, near start]{uu}{}
\ar{ur}
\end{tikzcd}\] that we now proceed to construct. In the front face, we claim that the top terms participate in equivalences
\(
K(R ; {\Bbb Q}_p) \simeq {\rm KH}(R ; {\Bbb Q}_p)\text{ and }K(R / p ; {\Bbb Q}_p) \simeq {\rm KH}(R/p ; {\Bbb Q}_p). 
\)
For the first equivalence, this follows from $R$ being regular. For the second, we observe that since $(R,P)$ is vertical and log regular, the log structure is supported on the special fiber $R/p$, so that Proposition \ref{square.7} is applicable to $R/p$ by \cite[Theorems III.1.11.6, III.1.11.8]{Ogu}. Under these equivalences, the vertical arrows of the front face are induced by the usual cyclotomic trace, and the front face is thus cartesian by \cite[Theorem 4.36]{CMM}. The top face is cartesian by \cite[Theorem 2.9]{Par3}, while the bottom is cartesian by Proposition \ref{prop:tccartnolog}.
The left face is obtained by \cite[Theorem F]{Par2}, $K(R;\Q_p)\simeq \KH(R;\Q_p)$, and $K(R[1/p];\Q_p)\simeq \KH((R,P);\Q_p)$.
Hence we can construct the dotted arrow by taking pushouts, and then the back face is cartesian. 
\end{proof}

  For a $\cY$ vertical, regular, log regular formal log scheme over $\cO_K^\sharp$, the generic fiber $\cY_K$ is a rigid analytic variety over $K$.
See \cite[Theorem 3.5]{Mor16} for the continuous K-theory of rigid analytic varieties.
By \cite[Definition 3.1]{Mor16} and  \cite[Theorem 5.23]{CMM} (see \cite[Example 4.3]{AMMN}),
for every $p$-complete ring $R$ with bounded $p$-power torsion,
we have
\begin{equation}
\label{square.8.1}
K^\cts(R[1/p];\Z_p)
\simeq
K(R[1/p];\Z_p).
\end{equation}

\begin{cor}
\label{square.8}
Let $\cY$ be a finite-dimensional, Noetherian, vertical, regular, log regular formal log scheme over $\cO_K^\sharp$.
Then there exists a natural cartesian square
\[
\begin{tikzcd}
K^\cts(\cY_K;\Q_p)\ar[d]\ar[r]&
\KH(\cY_k;\Q_p)\ar[d]
\\
\HC^-(\cY/\cO_K^\sharp;\Q_p)\ar[r]&
\HP(\cY/\cO_K^\sharp;\Q_p)
\end{tikzcd}
\]
of spectra.
\end{cor}
\begin{proof}
This is an immediate consequence of Proposition \ref{prop:old-thm-A} and \eqref{square.8.1}.
\end{proof}

\section{The Hyodo-Kato Chern class}

Fix a pre-log ring $(k, P)$ with $k$ a perfect field of positive characteristic and $P$ a monoid with $\overline{P} = P / P^*$ fine and saturated. The purpose of this section is to introduce the Hyodo-Kato Chern class generalizing the crystalline Chern class.

\subsection{Hyodo--Kato cohomology}
\label{HK} Let $k$ be a perfect field of positive characteristic. For every log smooth fs log scheme $X$ saturated over $(k,P)$,
Hyodo and Kato defined a logarithmic version of the de Rham-Witt complex \cite[4.1]{HK}
$W\Omega_{X/(k,P)}$.
Its cohomology
\[
R\Gamma_{\HK}(X/W) = R\Gamma_\setale(X,W\Omega_{X/(k,P)})
\]
is called the (integral) \emph{Hyodo-Kato cohomology}. 
Here,
$\setale$ denotes the strict \'etale topology on $X$. According to \cite{HK}, it agrees with the log crystalline cohomology $R\Gamma_{\rm log crys}(X/W^0)$, where $W^0$ denotes the ring of Witt vectors over $k$, equipped with the canonical Teichm\"uller lift log structure (see \cite[Section 4]{HK}). We will not use the crystalline description in what follows.

\begin{rmk} Let us collect a number of relevant properties:

\begin{enumerate}

\item

By \cite[Corollary 4.5]{HK},
we have an equivalence
\begin{equation}
\label{HK.0.4}
W\Omega_{X/(k,P)}/p
\simeq
\Omega_{X/(k,P)}.
\end{equation}

\item By \cite[Theorem 4.31]{BLMP}, for every saturated log smooth morphism $X\to \Spec(k,\N)$,
we have an equivalence
\[
R\Gamma_\setale(X,W\Omega_{X/(k,\N)})\simeq R\Gamma_{\cPrism}(\ul{X},\partial X\oplus_\N \N_{\rm perf}),
\]
natural in $X$,
where $\N_\mathrm{perf}$ denotes the colimit $p$-perfection of $\N$.
Therefore we immediately deduce that for every vector bundle $\ul{\cE}\to \ul{X}$ of rank $r+1$, the maps in \cite[(3.12.2)]{BLMP} induce a projective bundle formula
\begin{align*}
\bigoplus_{i=0}^r R\Gamma_\setale(X,W\Omega_{X/(k,\N)})[-2i] \simeq & \bigoplus_{i=0}^r R\Gamma_{\cPrism}((\ul{X},\partial X \oplus_\N \N_{\rm perf})[-2i] 
\\
\xrightarrow{\simeq}&
R\Gamma_{\cPrism}(\P(\cE)_{\rm pf}) \simeq
R\Gamma_\setale(\P(\cE),W\Omega_{X/(k,\N)}),
\end{align*}
where $\P(\cE)$ and $\P(\cE)_{\rm pf}$ are the log schemes with the pullback log structure from $X$ and $(\ul{X},\partial X\oplus_\N\N_{\rm perf})$ and the equivalence in the middle is \cite[Lemma 3.13]{BLMP}.
\end{enumerate}
\end{rmk}
From (2) we also deduce the $(\P^n,\P^{n-1})$ invariance: since the Nygaard filtration is complete and exhaustive, we may show invariance on the graded pieces. Using the secondary filtration, it suffices to prove that
\begin{equation}\label{eq:box-HK-1}
    R\Gamma_\setale((\ul{X},\partial X\oplus_\N \N_{\rm perf}),\L_{-/k})\simeq R\Gamma_\setale((\P^n\times \ul{X},\partial (X\times (\P^n,\P^{n-1}))\oplus_\N \N_{\rm perf}),\L_{-/k}).
\end{equation}
As observed in the proof of \cite[Theorem 4.27]{BLMP}, for any saturated log smooth map $(k,\N)\to (A,M)$,
we have equivalences \[
\L_{(A,M)/(k,\N)}\simeq \L_{(A,M\oplus_\N \N_{\rm perf})/(k,\N_{\rm perf})}\simeq \L_{(A,M\oplus_\N \N_{\rm perf})/k}.
\]
From this, we deduce \eqref{eq:box-HK-1} from the $(\P^n,\P^{n-1})$ invariance of the cotangent complex \cite[Proposition 8.3]{BLPO}.

\begin{prop}
\label{HK.1}
Let $X$ be a strict smooth fs log scheme over $(k,P)$.
Then the natural map of complexes
\[
R\Gamma_\et(\ul{X},W\Omega_{\ul{X}/k})
\to
R\Gamma_\setale(X,W\Omega_{X/(k,P)})
\]
is an equivalence.
\end{prop}
\begin{proof}
The question is strict \'etale local on $X$,
so we reduce to the case of $(k,P)\to (A,P)$ with $A$ strictly local.
By \eqref{HK.0.4},
it suffices to show
\begin{equation}
\label{HK.1.1}
\Omega_{A/k}^q
\cong
\Omega_{(A,P)/(k,P)}^q,
\end{equation}
which follows from \cite[Proposition IV.1.2.15]{Ogu}.
\end{proof}

\begin{prop}
\label{HK.2}
The presheaf of complexes $R\Gamma_\setale(-,W\Omega_{-/(k,P)})$ on the category of log smooth fs log schemes saturated over $(k,P)$ is a strict smooth sheaf.
\end{prop}
\begin{proof}
By \eqref{HK.0.4},
it suffices to show that $R\Gamma_\setale(-,\Omega_{-/(k,P)})$ is a strict smooth sheaf of complexes.
This is a consequence of \cite[Theorem 2.3]{BLPO2}.
\end{proof}

\subsection{Chern orientations} 
We review the theory of Chern orientations in logarithmic motivic homotopy theory \cite[Section 7]{BPO2}.
For a scheme $S$,
let $\lSm_S$ denote the category of log smooth fs log schemes over $S$.
Let $\SmlSm_S$ denote the category of smooth log smooth fs log schemes over $S$,
i.e., the category of those fs log schemes $X$ over $S$ such that $X\to S$ is log smooth and $\ul{X}\to \ul{S}$ is smooth.
The logarithmic motivic homotopy category $\logSH(S)$ in \cite{BPO} admits the model
\[
\logSH(S)
\simeq
\Sp_{\P^1}(\Sh_\sNis(\SmlSm_S,\Sp)[(\P^\bullet,\P^{\bullet-1})^{-1}])
\]
by  \cite[Corollary 3.4.7]{BPO2}.
Here,
$\sNis$ denotes the strict Nisnevich topology,
and $(\P^\bullet,\P^{\bullet-1})$ denotes the family of projections $X\times (\P^m,\P^{m-1})\to X$ for all $X\in \SmlSm_S$ and integers $m\geq 1$.

Let $\bE$ be a homotopy commutative monoid in $\logSH(S)$,
where $S$ is a scheme.
The $\bE$-cohomology groups
\[
\bE^{i,j}(X)
:=
\Hom_{\logSH(S)}(\Sigma^\infty X_+,\Sigma^{i,j}\bE)
\]
form a natural graded ring $\bE^{*,*}(X)$.
Let $\unit$ denote the unit of $\bE^{*,*}(X)$.
Recall from \cite[Definition 7.1.1]{BPO2} that a Chern orientation on $\bE$ is a class $\bE^{2,1}(\P^\infty/\pt)$ whose restriction to $\P^1/\pt$ is the class $(\P^1/\pt)\otimes \unit$ of $\bE^{2,1}(\P^1/\pt)$.

\begin{prop}
Let $\bE$ be a homotopy commutative monoid in $\logSH(S)$,
where $S$ is a scheme.
Then there exists a one-to-one correspondence between the set of Chern classes of $\bE$ and the set of functors
\[
c_1\colon \Pic \to \bE^{2,1}
\]
from $\SmlSm_S$ to $\Fun(\Delta^1,\Set)$ modulo natural isomorphisms.
\end{prop}
\begin{proof}
A Chern orientation yields such a functor, see \cite[Definition 7.1.3]{BPO2}.
On the other hand, when a functor $c_1\colon \Pic \to \bE^{2,1}$ is given,
the image of $\cO(-1)\in \Pic(\P^\infty)$ yields a Chern orientation of $\bE$.
These two are inverses to each other by construction.
\end{proof}

Let $\bE$ be a Chern oriented homotopy commutative monoid in $\logSH(S)$,
where $S$ is a scheme.
Recall from \cite[Theorem 7.2.8]{BPO2} that there exists a natural projective bundle formula
\begin{equation}
\bigoplus_{i=0}^d
\bE^{*-i,*-2i}(X)
\cong
\bE^{*,*}(\P(\cE))
\end{equation}
for every rank $d+1$ vector bundle $\cE\to X$ in $\lSm_S$.
Using this,
\cite[Definition 7.3.1]{BPO2} provides a natural definition of higher Chern classes
\[
c_i(\cE)\in \bE^{2i,i}(X).
\]
The Whitney sum formula
\[
c_i(\cE)
=
\sum_{j=0}^i
c_j(\cE')c_{i-j}(\cE'')
\]
holds for every exact sequence
\[
0
\to
\cE'
\to
\cE
\to
\cE''
\to
0
\]
of vector bundles over $X$,
see \cite[Proposition 7.3.16]{BPO2}.

\begin{const}
\label{HK.4}
Let $(A,M)$ be a log smooth saturated pre-log $(k,P)$-algebra such that $P\to M$ is saturated.
From the cdga $W\Omega_{(A,M)/(k,P)}$,
we can construct a natural cdga $E_*$ such that the degree $i$ part $E_i$ is $W\Omega_{(A,M)/(k,P)}^{\geq i}$ for every integer $i$.
The projective bundle formula in Section \ref{HK} yields $c\colon \Sigma_{S^1} ^\infty \P^1 \to E_1[2]$ in $\Sh_\setale(\lSm_{(k,P)}^\sat,\Sp)$ such that the diagram in  \cite[Proposition 2.16]{BLMP} commutes for $X\in \lSm_{(k,P)}^\sat$,
where $\lSm_{(k,P)}^\sat$ denotes the category of log smooth fs log schemes saturated over $(k,P)$.
By adapting \cite[Proof of Proposition 2.16]{BLPO} to $\Sh_\setale(\lSm_{(k,P)}^\sat,\Sp)$,
we obtain
\[
(W\Omega_{/(k,P)}^{\geq 0},
W\Omega_{/(k,P)}^{\geq 1},
\cdots)
\in
\CAlg(\Sp_{\P^1}(\Sh_\setale(\lSm_{(k,P)}^\sat,\Sp))).
\]
For $X\in \lSm_{(k,P)}^\sat$,
by restriction,
we obtain
\[
(W\Omega_{(-)\times_{\ul{X}}X/(k,P)}^{\geq 0},
W\Omega_{(-)\times_{\ul{X}}X/(k,P)}^{\geq 1},
\cdots)
\in
\CAlg(\Sp_{\P^1}(\Sh_\setale(\SmlSm_{\ul{X}},\Sp))).
\]
This is $(\P^m,\P^{m-1})$-invariant for every integer $m\geq 1$ due to Section \ref{HK}.
Hence we obtain
\[
(W\Omega_{(-)\times_{\ul{X}}X/(k,P)}^{\geq 0},
W\Omega_{(-)\times_{\ul{X}}X/(k,P)}^{\geq 1},
\cdots)
\in
\CAlg(\logSH(\ul{X})).
\]

This admits a Chern orientation given by the image of $\cO(-1)\in \Pic(\P_k^\infty)$ under the composite
\[
\Pic(\P_k^\infty) \to H_\et^2(\P_k^\infty,W\Omega_{\P_k^\infty/k})
\to
H_\setale^2(\P_X^\infty,W\Omega_{\P_X^\infty/(k,P)}),
\]
where the first map is the crystalline Chern class constructed by Gros \cite[Th\'eor\'eme III.1.1.1, Section III.2.1]{Gros}.
Then we obtain higher Chern classes $c_i(\cE)\in H_\setale^{2i}(X,W\Omega_{X/(k,P)})$ for every vector bundle $\cE\to X$ and integer $i$ satisfying the Whitney sum formula.
Using the usual formula \cite[\href{https://stacks.math.columbia.edu/tag/02UM}{Tag 02UM}]{stacks-project},
we obtain the \emph{Hyodo-Kato Chern character} $\ch_\HK(\cE)\in \bigoplus_{i=0}^\infty H_\setale^{2i}(X,W\Omega_{X/(k,P)})$.
From this, we also obtain
\begin{equation}
\label{HK.4.1}
\ch_\HK
\colon
K_0(\ul{X})
\to
\bigoplus_{i=0}^\infty H_\setale^{2i}(X,W\Omega_{X/(k,P)}).
\end{equation}
\end{const}

\begin{rmk}
\label{HK.5}
Let $X$ be a strict smooth fs log scheme over $(k,P)$,
and let $\cE\to X$ be a vector bundle.
By construction,
for every integer $i$,
the crystalline Chern class
\[
c_i(\ul{\cE})\in H_\et^{2i}(\ul{X},W\Omega_{\ul{X}/k})
\]
corresponds to the Hyodo-Kato Chern class
\[
c_i(\cE)\in H_\setale^{2i}(X,W\Omega_{X/(k,P)})
\]
under the isomorphism $H_\et^{2i}(\ul{X},W\Omega_{\ul{X}/k})\cong H_\setale^{2i}(X,W\Omega_{X/(k,P)})$ obtained by Proposition \ref{HK.1}.
\end{rmk}
\begin{rmk}
Let $Y$ be a strict normal crossing variety over $k$. Its rigid cohomology $R\Gamma_{\rm rig}(Y/K)$ is a complex of $K$-vector spaces, and it is $\mathbb{A}^1$-motivic, in the sense that it satisfies $\mathbb{A}^1$-invariance, étale descent, and $\P^1$-stability. As such, we obtain a representing object $E_{\rm rig} \in \CAlg(\mathrm{SH}_{\et}(k, \mathrm{Mod}_K))$ (see \cite[Theorem 1.21]{ayoub-weil}). This ring object is automatically $\mathrm{KGL}$-oriented (in fact, $R\Gamma_{\rm rig}$ is representable in $\mathrm{DM}(k)$), and as such we obtain a rigid Chern character
\[\ch_{\rm rig}\colon K_0(Y) \to \bigoplus_{i=0}^\infty H^{2i}_{\rm rig}(Y/K).\]
\end{rmk}
\begin{rmk}\label{rmk:rig_vs_logrig}
    Let $Y$ be again a proper semistable scheme over $k$, and write $X$ for the scheme $Y$ seen as a log smooth log scheme over the standard log point $(k,\mathbb{N})$. By e.g., \cite[3.11]{GK}, we can identify $H^*_{\mathrm{log}\text{-}\mathrm{rig}}(X)$ with the rational log crystalline cohomology $H^*_{\mathrm{log}\text{-}\mathrm{crys}}(X/W) \otimes K$, which is computed as $R\Gamma_\setale(X,W\Omega_{X/(k,P)})\otimes_{K_0} K = R\Gamma_{\HK}(X/W)\otimes_{K_0}K$, where $K_0=W(k)[1/p]$. 
    
    Again by Gro\ss e-Kl\"onne \cite[Section 1]{GK}, there is a natural morphism \[\phi\colon R\Gamma_{\rm rig}(Y/K)\to R\Gamma_{\mathrm{log}\text{-}\mathrm{rig}}(X/K)\]
induced by the natural morphism of log schemes $f\colon X \to  \ul{X} = (Y, \mathcal{O}^\times)$, which acts as the identity on the underlying scheme $Y$ and consequently induces the identity map on the sheaf of units $\mathcal{O}_Y^\times$. The first Chern class $c_1(L)$ of a line bundle $L$ is defined in both theories via the boundary map of the exponential sequence associated with $\mathcal{O}^\times$ (represented explicitly by the map $d\log\colon \mathcal{O}^\times \to \Omega^1$); because the morphism $f$ respects the sheaf $\mathcal{O}^\times$ and is compatible with the derivation $d$, it sends the rigid representative $d\log \tilde{g}_{ij}$ (where $\tilde{g}_{ij}$ is an overconvergent lift of the transition function) to the log crystalline representative $d\log [g_{ij}]$ (where $[g_{ij}]$ is the Teichmüller lift in the log de Rham-Witt complex $W\Omega_{-/(k, \mathbb{N)}}$), ensuring that $\phi(c_1^{\mathrm{rig}}(L)) = c_1^{\log}(L)$. This implies commutativity for all higher Chern classes via the splitting principle. In particular, we get a commutative triangle:
\[
\begin{tikzcd}
    K_0(Y; \Q) \arrow[r, "\ch_{\rm rig}"] \arrow[dr, "\ch_\HK"']&  \bigoplus_{i=0}^\infty H^{2i}_{\rm rig}(Y/K)\arrow[d, "\phi"]\\
    & \bigoplus_{i=0}^\infty H_\setale^{2i}(X,W\Omega_{X/(k,\N)})\otimes_{K_0} K
\end{tikzcd}
\]
Using again the identification between log rigid and Hyodo-Kato cohomology, by \cite[Corollary 4.70]{BGV} we have $N\circ \phi =0$, where $N$ is the monodromy operator. In particular, for every $\alpha \in K_0(Y)$, we have $N(\ch_{\HK}(\alpha))=0$.
\end{rmk}
\begin{rmk}
    The fact that $N\circ \ch_\HK =0$ can also be deduced more directly. Indeed, we have $N\circ c_1([L])=0$ for any line bundle $L$ by \cite[Proposition 2.2]{Yam11}. Since $N$ satisfies the Leibniz rule (it acts as a derivation on the algebra structure of cohomology, see \cite[Remark 3.8]{BGV}) and higher Chern classes  are polynomials in $c_1$, it follows that $N(\ch_{\HK}(\alpha))=0$ for any $\alpha \in K_0(Y)$.
\end{rmk}

\section{Splitting the \texorpdfstring{${\rm HKR}$}{HKR}-style filtrations} We now aim to obtain natural splittings of the rationalized filtrations of ${\rm HC}^-$ and ${\rm HP}$, in analogy with e.g.\ \cite[5.1.12]{Lod}. Throughout this section,
let $K$ be a complete discrete valuation field of characteristic $0$ with ring of integers $\cO_K$ and perfect residue field $k$ of characteristic $p$.
Let $K_0:=W(k)[1/p]$ and $\cO_K^\sharp:=(\cO_K,\cO_K-\{0\})$.

A \emph{semistable formal log scheme $\cY$ over $\cO_K^\sharp$} means a formal log scheme over $\cO_K^\sharp$ such that locally on $\cY$,
there exists a semistable fs log scheme $Y$ over $\cO_K^\sharp$ such that $\cY$ is isomorphic to the $p$-adic completion of $Y$.

Let $\cY_k$ denote the special fiber of $\cY$,
which is a log smooth fs log scheme over $(k,\N)$.
Let $\cY_K$ denote the rigid analytic generic fiber of $\cY$.

By \cite[Theorem 2]{Ogu95} (see the discussion at  \cite[p. 210]{Ogu95}),
there exists a natural equivalence, depending on the choice of a uniformizer $\pi$ of $\mathcal{O}_K$, 
\begin{equation}
\label{rig.0.2} 
R\Gamma_\dR(\cY/\cO_K^\sharp)
\simeq
R\Gamma_\setale(\cY_k,W\Omega_{\cY_k/(k,\N)})\otimes_{K_0} K,
\end{equation}
which is called the Hyodo-Kato isomorphism (or Hyodo-Kato equivalence).

Let $Y$ be a \emph{proper} semistable fs log scheme over $\cO_K^\sharp$.
Then we also have
\begin{equation}
\label{rig.0.3}
R\Gamma_\dR((Y_p^\wedge)/\cO_K^\sharp)
\simeq
R\Gamma_\dR(Y_K/K),
\end{equation}
where $Y_K$ denotes the generic fiber of $Y$, and $Y_p^\wedge$ denotes the $p$-completion of $Y$.
Note that \eqref{rig.0.3} is compatible with the Hodge filtrations.
Combining \eqref{rig.0.2} and \eqref{rig.0.3},
we obtain a natural equivalence (see \cite[Theorem 3]{Ogu95})
\[
\rho_\pi\colon R\Gamma_\dR(Y_K)
\simeq
R\Gamma_\setale(Y_k,W\Omega_{Y_k/(k,\N)})\otimes_{K_0} K,
\]
which is also called the Hyodo-Kato equivalence.
Here, $Y_k$ denotes the special fiber of $Y$.

\begin{prop}
\label{rig.3}
Let $\cY$ be a semistable formal log scheme over $\cO_K^\sharp$.
Then we have natural decompositions
\begin{gather*}
\HP(\cY/\cO_K^\sharp;\Q_p)
\simeq
\prod_{i\in \Z} R\Gamma_\dR(\cY/\cO_K^\sharp)[2i],
\\
\HC^-(\cY/\cO_K^\sharp;\Q_p)
\simeq
\prod_{i\in \Z} \Fil^{\geq i} R\Gamma_\dR(\cY/\cO_K^\sharp)[2i],
\end{gather*}
where $\Fil^{\geq i}R\Gamma_\dR(\cY/\cO_K^\sharp)
:=
R\Gamma_\setale(\cY,\Omega_{\cY/\cO_K^\sharp}^{\geq q})$.
\end{prop}
\begin{proof}
This is analogous to \cite[Construction 4.11]{AMMN}: Replacing the classical HKR-theorem with its logarithmic analog \cite[Theorem 1.1]{BLPO} and using continuity of ${\rm HC}^-$ and ${\rm HP}$ \cite[Proposition 4.2]{AMMN}, we may argue exactly as in the non-logarithmic case (using Adams operations as in \cite[Section 9.4]{BMS2}). 
\end{proof}

\begin{thm}
\label{rig.4}
Let $X$ be a proper flat scheme over $\cO_K$ with semistable special fiber $X_k$.
Then there exists a natural cartesian square
\[
\begin{tikzcd}
K^\cts(X;\Q)\ar[d]\ar[r]&
K(X_k;\Q)\ar[d]
\\
\prod_{i\in \Z} \Fil^{\geq i}R\Gamma_\dR(X_K/K)[2i]\ar[r]&
\prod_{i\in \Z}  R\Gamma_\dR(X_K/K)[2i]
\end{tikzcd}
\]
of spectra,
where $X_K$ denotes the generic fibers of $X$.
\end{thm}
\begin{proof}
By \cite[Corollary IV.3.1.18]{Ogu},
there exists a semistable fs log scheme $Y$ over $\cO_K^\sharp$ such that $\ul{Y}\cong X$.
Combine Corollary \ref{cor:beilfibernologktheory}, \eqref{rig.0.3}, and Proposition \ref{rig.3}.
\end{proof}

\begin{thm}
\label{rig.5}
Let $Y$ be a semistable fs log scheme over $\cO_K^\sharp$.
Then there exists a natural cartesian square
\[
\begin{tikzcd}
K^\cts(Y_K;\Q_p)\ar[d]\ar[r]&
\mathrm{KH}(Y_k;\Q_p)\ar[d]
\\
\prod_{i\in \Z}\Fil^{\geq i} R\Gamma_\dR(Y_K/K)[2i]\ar[r]&
\prod_{i\in \Z}  R\Gamma_\dR(Y_K/K)[2i]
\end{tikzcd}
\]
of spectra.
\end{thm}
\begin{proof}
Combine Corollary \ref{square.8}, \eqref{rig.0.3}, and Proposition \ref{rig.3}.
\end{proof}

\section{The Hyodo--Kato Chern character via the trace}\label{sec:hkchern} Keep the notation of the previous section. We now compare the Hyodo-Kato Chern character with a construction obtained by the cyclotomic trace.
As a consequence,
we will identify the obstruction class for the $p$-adic deformation problem with the Hyodo-Kato Chern character.

\begin{const}
\label{comp.1}
Let $\cY$ be a semistable formal log scheme over $\cO_K^\sharp$.
Let $\mathrm{trc}_\HK$ denote the composite
\[
K(\ul{\cY}_k;\Q)
\to
\HP(\cY/\cO_K^\sharp;\Q_p)
\xrightarrow{\simeq}
\prod_{i\in \Z} R\Gamma_{\dR}(\cY/\cO_K^\sharp)[2i],
\]
where the first map is obtained by the right vertical map of \eqref{square.4.1},
and the second map is obtained by Proposition \ref{rig.3}.
For $\lambda\in K^\times$, let $\mathrm{trc}_\HK^\lambda$ be the composite of $\mathrm{trc}_\HK$ and the automorphism of $\prod_{i\in \Z} R\Gamma_{\dR}(\cY_K/K)[2i]$ multiplying by $\lambda^i$ for each factor.
\end{const}

\begin{prop}
\label{comp.2}
There exists $\lambda\in K^\times$ such that for every semistable formal log scheme $\cY$ over $\cO_K^\sharp$,
the Hyodo-Kato Chern character
\[
\ch_\HK
\colon
K_0(\ul{\cY}_k)\to \prod_{i\in \Z} H_{\setale}^{2i}(\cY_k,W\Omega_{\cY_k/(k,\N)})\otimes_{K_0}K
\]
agrees with $\pi_0(\mathrm{trc}_\HK^\lambda)$ under the Hyodo-Kato isomorphism $\rho_\pi$.
\end{prop}
\begin{proof}
For every scheme $X$ and integer $n\geq 0$,
let $\Vect_n(X)$ denote the groupoid of $n$-dimensional vector bundles over $X$.
From $\mathrm{trc}_\HK^\lambda$ and \eqref{rig.0.2},
we obtain a natural map of spaces
\[
f_n^\lambda
\colon
\Vect_n(\ul{\cY}_k)
\to
\Omega_{S^1}^\infty
\big(
\prod_{i\in \Z}
R\Gamma_\setale(\cY_k,W\Omega_{\cY_k/(k,\N)})\otimes_{K_0}K[2i]
\big).
\]
The left-hand side is a smooth sheaf since $\Vect_n$ is represented by the quotient Artin stack $\BGL_{n,k}$, and the right-hand side is a smooth sheaf by Proposition \ref{HK.2}.
Hence we can define $f_n^\lambda$ for $\cY=\BGL_{n,\cO_K^\sharp}$.
Using the naturality of $f_n^\lambda$,
it suffices to show that $f_n^\lambda$ sends the tautological bundle $\cE$ of $\BGL_{n,k}$ to the Hyodo-Kato Chern character of $\cE$ for some $\lambda\in K^\times$.
By Remark \ref{HK.5},
we reduce to the analogous claim for the crystalline Chern character,
and this is already checked in \cite[Proof of Proposition 4.12]{AMMN}.
\end{proof}

Now, we prove Theorem \ref{thm:hkchern}:

\begin{thm}\label{thm:hkchern2} Let $X$ be a proper flat scheme over $\cO_K$ with semistable special fiber $X_k$.
Then a class $x\in K_0(X_k;\Q)$ lifts to $K_0^\cts(X;\Q)$ if and only if the Hyodo-Kato Chern character $\ch_\HK(x)$ is contained in $\bigoplus_{i= 0}^\infty \Fil^{\geq i}H_\dR^{2i}(X_K;\Q_p)$ under the Hyodo-Kato equivalence.
\end{thm}
\begin{proof}
Let us first spell out what $\ch_\HK$ means in this case.
Consider $Y$ in the proof of Theorem \ref{rig.4}.
Then \eqref{HK.4.1} for $Y_k$ is
\[
\ch_\HK\colon K_0(X_k)
\to
\bigoplus_{i=0}^\infty H_{\setale}^{2i}(Y_k,W\Omega_{Y_k/(k,P)}).
\]

By Theorem \ref{rig.4},
$x$ lifts to $K_0^\cts(X;\Q)$ if and only if $\pi_0(\mathrm{trc}_\HK)(x)$ is contained in $\bigoplus_{i= 0}^\infty \Fil^{\geq i}H_\dR^{2i}(X_K/K)$.
Use Proposition \ref{comp.2} to conclude the proof. 
\end{proof}

\begin{rmk}Let $X$ be a proper flat scheme over $\cO_K$.
Due to Beilinson \cite{Bei2}, \cite[Theorem 4.13]{AMMN},
given $x\in K_0(X_k;\Q_p)$,
we know that there exists a natural obstruction class
\[
o(x)\in \bigoplus_{i=0}^\infty H_\dR^{2i}(X_K) / \Fil^{\geq i}H_\dR^{2i}(X_K)
\]
such that $o(x)$ vanishes if and only if $x$ lifts to $K_0^\cts(X;\Q_p)$.
Theorem \ref{thm:hkchern} identifies $o(x)$ with the Hyodo-Kato Chern character under the assumption that $X$ has semistable reduction.
\end{rmk}

\begin{rmk}
Let $X$ be a proper flat scheme over $\cO_K$ with semistable special fiber $X_k$.
Gregory and Langer \cite[Theorem 1.1]{GL} showed an analogous statement for their log-motivic cohomology $\mathbb{H}^{2n}(X_k,\Z_{\mathrm{log},X_k}(n))\otimes \Q$ under the assumptions $\cO_K=W(k)$ and $n<p$.
As stated in \cite[Remark 4.11]{GL},
their result does not reprove Yamashita's result \cite[Theorem 0.1]{Yam11}.
However, ours does reprove it,
see Theorem \ref{Pic.3}.
\end{rmk}

\begin{rmk}[On the Hyodo-Kato isomorphism] The obstruction to lifting $K_0$-classes provided by Theorem \ref{thm:logbek} makes implicit use of the Hyodo-Kato isomorphism for the identification between (rational) Hyodo-Kato cohomology and the de Rham cohomology of the generic fiber. Thus, a priori, the condition might depend on the choice of a uniformizer $\pi$. 

This is however not the case, as remarked already by Yamashita in \cite[Corollary 2.3]{Yam11} for the case of line bundles. Indeed, by \cite[Theorem 5.1]{HK} (see also \cite[Remark 4.18]{BGV}) for any $u\in \mathcal{O}_K^\times$, we have $\rho_{u\pi} = \rho_\pi \circ \exp(\log(u)N)$.
Since $N$ is nilpotent, $\exp(\log(u) N)$ is a finite sum $\Sigma_{i=0}^m \log(u)^i N^i/i!$, where $\log(u)\in W(k)[1/p]$ is the $p$-adic logarithm of $u$.
Furthermore,
$N\circ \ch_{\HK}=0$ by Remark \ref{rmk:rig_vs_logrig}.
Hence we conclude that for any
$\alpha\in K_0(X_k)$, the class $\rho_\pi(\ch_\HK(\alpha))$ in $\bigoplus_{i=0}^\infty H^{2i}_{\dR}(X_K)$ does not depend on the choice of $\pi$. 
\end{rmk}

\section{The semistable Lefschetz \texorpdfstring{$(1,1)$}{(1,1)}-Theorem}\label{sect:Yamashita}
Our next purpose is to show that Theorem \ref{thm:hkchern} recovers Yamashita's semistable generalization \cite[Theorem 0.1]{Yam11} of the Berthelot-Ogus theorem \cite[Theorem 3.8]{BO}.
We refer to \cite[Section 2]{Yam11} for the Hyodo-Kato first Chern class (also called log-crystalline first Chern class)
\[
c_1
\colon
\Pic(Z)
\to
H_{\setale}^{2}(Z,W\Omega_{\cY_k/(k,\N)})
\]
for every semistable fs log scheme $Z$ over $(k,\N)$,
where
\[
\Pic(Z):=H_\setale^1(Z,\cM_Z^\mathrm{gp}).
\]
A key property of $c_1$ in \cite[Proposition 2.4]{Yam11} is that for every semistable fs log scheme $Y$ over $\cO_K^\sharp$,
there is a natural commutative diagram
\begin{equation}
\label{Pic.0.1}
\begin{tikzcd}
\Pic(Y_K)\ar[d,"c_1"']\ar[r,"\cong"]&
\Pic(Y)\ar[r]&
\Pic(Y_k)\ar[d,"c_1"]
\\
H_\dR^2(Y_K/K)\ar[rr,"\cong"]&
&
H_{\setale}^{2}(Y_k,W\Omega_{Y_k/(k,\N)})\otimes_{K_0} K,
\end{tikzcd}
\end{equation}
where the left vertical map is the de Rham first Chern class.

\begin{prop}
\label{Pic.4}
There exists $\mu\in K^\times$ such that for every semistable $p$-adic formal log scheme $\cY$ over $\cO_K^\sharp$,
the induced diagram
\[
\begin{tikzcd}
\Pic(\ul{\cY}_k)\ar[r]\ar[d]&
K_0(\ul{\cY}_k)\ar[rr,"\pi_0(\mathrm{trc}_\HK^\mu)"]&
&
\prod_{i\in \Z} H_{\setale}^{2i}(\cY_k,W\Omega_{\cY_k/(k,\N)})\ar[d,"p_1"]
\\
\Pic(\cY_k)\ar[rrr,"\cong"]&
&
&
H_{\setale}^{2}(\cY_k,W\Omega_{\cY_k/(k,\N)})\otimes_{K_0}K
\end{tikzcd}
\]
commutes,
where $p_1$ denotes the projection.
\end{prop}
\begin{proof}
As in the proof of Proposition \ref{comp.2},
we reduce to checking the commutativity for the class $[\cE]\in \Pic(\mathrm{B}\mathbb{G}_{m,k})$ with $\cY=\mathrm{B}\mathbb{G}_{m,\cO_K^\sharp}$,
where $\cE$ is the tautological line bundle.
We have
\[
H_{\setale}^{2}(\cY_k,W\Omega_{\cY_k/(k,\N)})\otimes_{K_0}K
\cong
H_\dR^2(\mathrm{B}\mathbb{G}_{m,K})
\cong
H_\dR^2(\P^\infty_K)
\cong
K.
\]
This shows the existence of $\mu\in K$.
Proposition \ref{comp.2} ensures $\mu\neq 0$.
\end{proof}

For abbreviation,
we set $\Pic(-;\Q):=\Pic(-)\otimes \Q$.
\begin{prop}
\label{Pic.1}
Let $X$ be a proper flat scheme over $\mathcal{O}_K$ with semistable special fiber $X_k$.
Then a class $x\in \mathrm{Pic}(X_k;\Q)$ lifts to $\mathrm{Pic}(X;\Q)$ if the first Hyodo-Kato Chern class $c_1(x)$ is contained in $\mathrm{Fil}^{\geq 1} H_\mathrm{dR}^2(X_K;\Q_p)$ under the Hyodo-Kato equivalence.
\end{prop}
\begin{proof}
Regard $x$ as a class of $K_0(X_k;\Q)$.
Theorem \ref{thm:hkchern} allows us to lift $x$ to $y\in K_0^\mathrm{cts}(X;\Q)$.
Consider
\[
\det(y)\in (\lim_n \mathrm{Pic}(X/p^n))\otimes \Q
\cong
\mathrm{Pic}(X;\Q),
\]
where the isomorphism is a consequence of the Grothendieck existence theorem \cite[Corollaire 5.1.6]{EGA31}.
Since $\det$ is compatible with the pullback maps for $K$ and $\mathrm{Pic}$,
$\det(y)$ is a lift of $\det(x)=x$.
\end{proof}

\begin{prop}
\label{Pic.2}
Let $i\colon Z\to Y$ be a strict closed immersion of fs log schemes.
Assume that $Y-Z$ has a trivial log structure.
Then the induced square
\[
\begin{tikzcd}
R\Gamma_\et
(\underline{Y},\cO_{\ul{Y}}^\times)\ar[d]\ar[r]&
R\Gamma_\et
(\underline{Z},\cO_{\ul{Z}}^\times)\ar[d]
\\
R\Gamma_\setale
(Y,\cM_Y^\mathrm{gp})\ar[r]&
R\Gamma_\setale
(Z,\cM_Z^\mathrm{gp})
\end{tikzcd}
\]
is cartesian.
\end{prop}
\begin{proof}
Taking fibers along the horizontal maps yields
\[
R\Gamma_\et
(\underline{Y},\ker(\mathcal{O}_{\underline{Y}}^\times\to \underline{i}_*\mathcal{O}_{\underline{Z}}^\times))
\to
R\Gamma_\setale
(Y,\ker(\mathcal{M}_Y^\mathrm{gp}\to i_*\mathcal{M}_{Z}^\mathrm{gp})).
\]
Hence for every point $y\in Y$,
it suffices to show
\[
\ker(\mathcal{O}_{\underline{Y},y}^\times
\to
(\underline{i}_*\mathcal{O}_{\underline{Z}}^\times)_y)
\cong
\ker(\mathcal{M}_{Y,y}^\mathrm{gp}
\to
(\underline{i}_*\mathcal{M}_{\underline{Z}}^\mathrm{gp})_y).
\]

If $y\notin Z$,
then $(\underline{i}_*\mathcal{O}_{\underline{Z}}^\times)_y$ and $(\underline{i}_*\mathcal{M}_{\underline{Z}}^\mathrm{gp})_y$ vanish,
and $\mathcal{O}_{\underline{Y},y}^\times\cong \mathcal{M}_{Y,y}^\mathrm{gp}$ since $Y-Z$ has a trivial log structure.
This shows the claim.

If $y\in Z$,
then $\cM_{Y,y}/\mathcal{O}_{\underline{Y},y}^\times\cong \cM_{\underline{Z},y}/\mathcal{O}_{\underline{Z},y}^\times$ since $i$ is strict.
Apply the group completion on both sides to conclude.
\end{proof}

Now, we reprove Yamashita's theorem \cite[Theorem 0.1]{Yam11}.

\begin{thm}
\label{Pic.3}
Let $Y$ be a vertical log smooth fs log scheme saturated over $\mathcal{O}_K^\sharp$.
Then a class $x\in \mathrm{Pic}(Y_k;\Q)$ lifts to $\mathrm{Pic}(Y;\Q)$ if and only if the first Hyodo-Kato Chern class $c_1(x)$ is contained in $\mathrm{Fil}^{\geq 1} H_\mathrm{dR}^2(Y_K/K)$ under the Hyodo-Kato equivalence.
\end{thm}
\begin{proof}
The image of the de Rham first Chern class $c_1\colon \Pic(Y_K)\to H_\dR^2(Y_K/K)$ is contained in $\Fil^{\geq 1}H_\dR^2(Y_K/K)=H_{\et}(Y_K,\Omega_{Y_K/K}^{\geq 1})$ since the de Rham first Chern class is defined using $d\log \colon \cO^\times\to \Omega_{/K}^1$.
Hence the only if direction follows from \eqref{Pic.0.1}.

For the if direction,
consider the exact sequence
\[
\mathrm{Pic}(\underline{Y};\Q)
\to
\mathrm{Pic}(Y;\Q)
\oplus
\mathrm{Pic}(\underline{Y_k};\Q)
\to
\mathrm{Pic}(Y_k;\Q)
\to
0
\]
obtained by Proposition \ref{Pic.2}.
Choose $y\in \mathrm{Pic}(Y;\Q)$ and $z\in \mathrm{Pic}(\underline{Y_k};\Q)$ such that $(y,z)$ maps to $x$.
There exists a lift $w\in \mathrm{Pic}(\underline{Y};\Q)$ of $z$ by Propositions \ref{Pic.4} and \ref{Pic.1}.
Let $u$ be the image of $w$ in $\mathrm{Pic}(Y;\Q)$.
Then $y-u$ is a lift of $x$.
\end{proof}

This also indicates that Theorem \ref{rig.5} is a reasonable $K$-theoretic generalization of \cite[Theorem 0.1]{Yam11}.

\section{Characteristic 0 case}\label{sec.char0}

Let $K$ be an algebraic extension of $\Q$ throughout this section.
Given a smooth proper scheme $X$ over $K[[t]]$ and a class $x\in K_0(X/t)$, Morrow \cite[Theorem 1.1]{Mor14} showed that $x$ lifts to $\lim_r \pi_0 K(X/t^r)$ if and only if the de Rham class belongs to the flat filtration in the sense of \cite[Section 3]{Mor14},
where $X/t^r:=X\otimes_{K[[t]]}K[[t]]/t^r$.
This may be regarded as a characteristic $0$ analogue of \cite{BEK14}.
The purpose of this section is to provide a semistable generalization of this result.

Let $L_\mathrm{cdh}\Omega_{-/K}^q$ denote the cdh sheafification of the presheaf of complexes (not abelian groups) $\Omega_{-/K}^q$ on the opposite category of $K$-algebras.
By \cite[Theorem 3]{HJ} (see \cite[Theorem 1.2 (3)]{HKK}), if $A$ is a smooth $K$-algebra we have a quasi-isomorphism of complexes
\[
\Omega_{A/K}^q
\simeq L_\mathrm{cdh}\Omega_{A/K}^q.
\]

\begin{prop}
\label{A.1}
For every semistable $(K[t],\langle t\rangle)$-algebra $(A,P)$ and integer $q$, $L_\mathrm{cdh}\Omega_{(A/t)/K}^{q}$ is discrete,
and we have an induced homotopy cartesian square
\begin{equation}
\label{A.1.1}
\begin{tikzcd}
\Omega_{A/K}^q\ar[d]\ar[r]&
L_\mathrm{cdh}\Omega_{(A/t)/K}^{q}\ar[d]
\\
\Omega_{(A,P)/K}^q\ar[r]&
\Omega_{(A/t,P)/K}^q.
\end{tikzcd}
\end{equation}
\end{prop}
\begin{proof}
Let $\mathrm{Spec}(A_1),\ldots,\mathrm{Spec}(A_m)$ be the irreducible components of $\mathrm{Spec}(A/t)$.
For a nonempty subset $I:=\{i_1,\ldots,i_s\}\subset \{1,\ldots,m\}$,
we set $A_I:=A_{i_1}\otimes_A \cdots \otimes_A A_{i_s}$.
We claim
\begin{equation}
\label{A.1.2}
\Omega_{(A/t,P)/K}^*
\simeq
\lim_{I\subset \{1,\ldots,m\},I\neq \emptyset}
\Omega_{(A_I,P)/K}^*,
\end{equation}
where $\lim$ is the homotopy limit.
This question is Zariski local on $A$,
so we may assume that there exists a strict \'etale map of $(K[t],\langle t\rangle)$-algebras
\[
(K[x_1,\ldots,x_n],\langle x_1,\ldots,x_d\rangle)\to (A,P)
\]
for some integers $0\leq d\leq n$.
By taking out the part $\Omega_{K[x_{d+1},\ldots,x_n]/K}^*$,
we reduce to the case where $d=n$.
In this case,
we set $A_i:=K[x_1,\ldots,x_n]/(x_i)$.
For a nonempty subset $I:=\{i_1,\ldots,i_s\}\subset \{1,\ldots,n\}$,
the coefficient of $\frac{dx_1}{x_1}\cdots \frac{dx_q}{x_q}$ with an integer $q\geq 0$ in $\Omega_{(A_I,P)/K}^*$ is
\[
K[x_1,\ldots,x_n]/(x_{i_1},\ldots, x_{i_s})
\]
and in $\Omega_{(A/t,P)/K}^*$ is
\[
K[x_1,\ldots,x_n]/(x_1\cdots x_n).
\]
Hence without loss of generality,
it suffices to show
\[
K[x_1,\ldots,x_n]/(x_1\cdots x_n)
\simeq
\lim_{I=\{i_1,\ldots,i_s\}\subset \{1,\ldots,n\},I\neq \emptyset}
K[x_1,\ldots,x_n]/(x_{i_1},\ldots, x_{i_s}).
\]
This is simple algebra.

As a consequence,
we obtain a natural map
\[
L_\mathrm{cdh}\Omega_{(A/t)/K}^{q}
\to
\Omega_{(A/t,P)/K}^q
\]
since $L_\mathrm{cdh}\Omega_{(A/t)/K}^{q}$ is equivalent to the homotopy limit
\[
\lim_{I\subset \{1,\ldots,m\},I\neq \emptyset}
\Omega_{A_I/K}^*.
\]
This constructs \eqref{A.1.1}.

Next, we claim that \eqref{A.1.1} is homotopy cartesian.
This question is Zariski local on $A$.
As above,
we reduce to the case where $(A,P)=(K[x_1,\ldots,x_n],\langle x_1,\ldots,x_n\rangle)$.
For a nonempty subset $I:=\{i_1,\ldots,i_s\}\subset \{1,\ldots,n\}$,
the coefficient of $dx_1\cdots dx_q$ in $\Omega_{A_I/K}^{q}$ is isomorphic to $K[x_1,\ldots,x_n]/(x_{i_1},\ldots,x_{i_s})$ if $I\subset \{q+1,\ldots,n\}$ and $0$ otherwise.
Hence the coefficient of $dx_1\cdots dx_q$ in $L_\mathrm{cdh}\Omega_{(A/t)/K}^{q}$ is equivalent to the homotopy limit
\[
\lim_{I=\{i_1,\ldots,i_s\}\subset \{q+1,\ldots,n\},I\neq \emptyset}
K[x_1,\ldots,x_n]/(x_{i_1},\ldots,x_{i_s}),
\]
which is equivalent to $K[x_1,\ldots,x_n]/(x_{q+1}\cdots x_n)$ as above.
Therefore,
without loss of generality,
it suffices to show that the square
\[
\begin{tikzcd}[column sep=small]
K[x_1,\ldots,x_n]\,dx_1 \cdots dx_q
\ar[d]\ar[r]&
K[x_1,\ldots,x_n]/(x_{q+1} \cdots x_n) \,dx_1 \cdots dx_q\ar[d]
\\
K[x_1,\ldots,x_n] \frac{dx_1}{x_1}\cdots  \frac{dx_q}{x_q}\ar[r]&
K[x_1,\ldots,x_n]/(x_1\cdots x_n) \frac{dx_1}{x_1}\cdots  \frac{dx_q}{x_q}
\end{tikzcd}
\]
is homotopy cartesian.
The horizontal maps are surjective,
so it suffices to show that the induced map of the kernels
\[
(x_{q+1}\cdots x_n)\, dx_1\cdots dx_q
\to
(x_1\cdots x_n)\tfrac{dx_1}{x_1}\cdots \tfrac{dx_q}{x_q}
\]
is an isomorphism,
which is clear.

Finally,
we deduce that $L_\mathrm{cdh}\Omega_{(A/t)/K}^{q}$ is discrete since the other three corners of \eqref{A.1.1} are discrete.
\end{proof}

For an inverse system $\cdots \to X_1\to X_0$ in a category or $\infty$-category,
let $\{X_r\}$ denote the associated pro-object.

\begin{prop}
\label{A.2}
For every semistable $(K[t],\langle t\rangle)$-algebra $(A,P)$ and integer $q$,
we have an induced homotopy cartesian square
\[
\begin{tikzcd}
\{\Omega_{(A/t^r)/K}^q\}_r\ar[d]\ar[r]&
L_\mathrm{cdh}\Omega_{(A/t)/K}^{q}\ar[d]
\\
\{\Omega_{(A/t^r,P)/K}^q\}_r\ar[r]&
\Omega_{(A/t,P)/K}^q.
\end{tikzcd}
\]
\end{prop}
\begin{proof}
As noted in \cite[Proof of Theorem 2.1]{Mor14},
we have
\[
\{\Omega_{(A/t^r)/K}^q\}
\cong
\{\Omega_{A/K}^q\otimes_A A/t^r\}
\]
as a consequence of $d(t^{2r})\subset t^r \Omega_{A/K}^1$.
We have a similar isomorphism for the case of $(A,P)$ too.
Hence it suffices to construct an induced homotopy cartesian square
\[
\begin{tikzcd}
\{\Omega_{A/K}^q\otimes_A A/t^r\}_r\ar[d]\ar[r]&
L_\mathrm{cdh}\Omega_{(A/t)/K}^{q}\ar[d]
\\
\{\Omega_{(A,P)/K}^q \otimes_A A/t^r\}_r\ar[r]&
\Omega_{(A/t,P)/K}^q.
\end{tikzcd}
\]
Arguing as in Proposition \ref{A.1.1},
it suffices to show that the square
\[
\begin{tikzcd}[column sep=small]
\{K[x_1,\ldots,x_n]/(x_1^r\cdots x_n^r)\,dx_1 \cdots dx_q\}_r
\ar[d]\ar[r]&
K[x_1,\ldots,x_n]/(x_{q+1} \cdots x_n) \,dx_1 \cdots dx_q\ar[d]
\\
\{K[x_1,\ldots,x_n]/(x_1^r\cdots x_n^r) \frac{dx_1}{x_1}\cdots  \frac{dx_q}{x_q}\}_r\ar[r]&
K[x_1,\ldots,x_n]/(x_1\cdots x_n) \frac{dx_1}{x_1}\cdots  \frac{dx_q}{x_q}
\end{tikzcd}
\]
is homotopy cartesian.
The horizontal maps are surjective,
so it suffices to show that the induced map of the kernels
\[
\{(x_{q+1}\cdots x_n)/(x_1^r\cdots x_n^r)\, dx_1\cdots dx_q\}_r
\to
\{(x_1\cdots x_n)/(x_1^r\cdots x_n^r)\tfrac{dx_1}{x_1}\cdots \tfrac{dx_q}{x_q}\}_r
\]
is an isomorphism of pro-abelian groups.
This holds since the cokernel vanishes and the kernel is Mittag-Leffler zero.
\end{proof}

Let $(A,P)$ be a semistable $(K[t],\langle t\rangle)$-algebra.
Consider the cdh sheafification of the HKR decomposition of $\mathrm{HC}^-(-)$ and use the assumption that $K$ is an algebraic extension of $\Q$ as in \cite[Proof of Theorem 3.2]{Mor14}
to obtain
\begin{equation}
\label{A.3.1}
L_\mathrm{cdh}\mathrm{HC}^-(A/t)
\simeq
\bigoplus_{p\in \Z} L_\mathrm{cdh} \Omega_{(A/t)/K}^{\geq p}[2p].
\end{equation}
We have the composite
\begin{align*}
\mathrm{trc}_\mathrm{dR}
\colon
\mathrm{KH}(A/t)
\xrightarrow{\mathrm{trc}} &
L_\mathrm{cdh}\mathrm{HC}^-(A/t)
\\
\xrightarrow{\simeq} &
\bigoplus_{p\in \Z} L_\mathrm{cdh} \Omega_{(A/t)/K}^{\geq p}[2p]
\to
\bigoplus_{p\in \Z} \Omega_{(A/t,P)/K}^{\geq p}[2p],
\end{align*}
see \cite[(4.1)]{EM23} for $\mathrm{trc}$.

Let $X$ be a semistable fs log scheme over $(K[[t]],\langle t\rangle)$ in the sense that $X$ is Zariski locally of the form $\mathrm{Spec}(A\otimes_{K[t]}K[[t]])$ for some semistable $(K[t],\langle t\rangle)$-algebra $(A,P)$.
After taking strict \'etale sheafification,
we obtain
\[
\mathrm{trc}_\mathrm{dR}
\colon
\mathrm{KH}(\underline{X}/t)
\to
\bigoplus_{p\in \Z}
R\Gamma_\setale(X/t,\Omega_{(X/t)/K}^{\geq p})[2p],
\]
where $X/t^r:=X\otimes_{K[[t]]}K[[t]]/t^r$ for every integer $r\geq 1$.
After taking $\pi_0$,
we also obtain
\[
\mathrm{trc}_\mathrm{dR}
\colon
\pi_0\mathrm{KH}(\underline{X}/t)
\to
\bigoplus_{p\in \Z}
H^{2p}_\setale((X/t)/K,\Omega_{(X/t)/K}^{\geq p}).
\]

\begin{prop}
\label{A.3}
Let $X$ be a semistable fs log scheme over $(K[[t]],\langle t\rangle)$.
Then for every integer $r\geq 1$,
there exists a natural isomorphism of pro-abelian groups
\[
\{\pi_n\mathrm{HC}^-(\underline{X}/t^r)\}_r
\cong
\{H^{2p-n}_\et(\underline{X},\bigoplus_{p\in \Z} \Omega^{\geq p}_{(\underline{X}/t^r)/K})\}_r.
\]
\end{prop}
\begin{proof}
If $K=\Q$,
then this is a special case of \cite[Corollary 2.2]{Mor14}.
The general case of $K$ follows from the assumption that $K$ is an algebraic extension of $\Q$.
\end{proof}

\begin{lem}
\label{A.7}
Let
\[
\begin{tikzcd}
\{A_r\}_r\ar[d]\ar[r]&
B\ar[d]
\\
\{C_r\}_r\ar[r]&
D
\end{tikzcd}
\]
be a commutative square of pro-$K$-complexes,
where $A_r$, $B$, $C_r$, and $D$ are $K$-complexes.
Assume that $\{\pi_* X_r\}_r\to \{\pi_* Y_r\}_r$ is an isomorphism of pro-$K$-vector spaces,
where $X_r:=\mathrm{cofib}(A_r\to B)$ and $Y_r:=\mathrm{cofib}(C_r\to D)$.
Assume also that $\pi_i D$ is finite dimensional for every $i$ and there is an isomorphism of pro-$K$-vector spaces $\{C_r\}_r \cong \{C_r'\}_r$ such that each $\pi_i C_r'$ is finite dimensional for every $i$.
Then the induced commutative diagram
\[
\begin{tikzcd}
\lim_r \pi_* A_r\ar[d]\ar[r]&
\pi_* B\ar[d]\ar[r]&
\lim_r \pi_{*-1} X_r\ar[d,"\cong"]
\\
\lim_r \pi_* C_r\ar[r]&
\pi_* D\ar[r]&
\lim_r \pi_{*-1} Y_r
\end{tikzcd}
\]
has exact rows.
\end{lem}
\begin{proof}
This is just a rephrasing of \cite[Part of the proof of Theorem 3.2]{Mor14}.
The assumption on $\pi_i C_r$ and $\pi_i D$ implies that the lower row is exact and there exists an isomorphism of pro-$K$-vector spaces $\{U_r\}_r\cong \{\pi_iY_r\}_r$ such that each $U_r$ is finite dimensional.
Choose an isomorphism of pro-abelian groups $\{\pi_i A_r\}\cong \{V_r\}_r$ with maps $\alpha_r\colon U_r\to V_r$ for all $r$ and a commutative square
\[
\begin{tikzcd}
\{\pi_i X_r\}_r \ar[d,"\cong"']\ar[r]&
\{\pi_i A_r\}_r\ar[d,"\cong"]
\\
\{U_r\}_r\ar[r]&
\{V_r\}_r.
\end{tikzcd}
\]
Since $\mathrm{im}(\alpha_r)$ is finite dimensional, $\{\mathrm{im}(\alpha_r)\}_r$ is Mittag-Leffler.
Using the Milnor sequence,
we see that
\[
0
\to
\lim_r \mathrm{im}(\alpha_r)
\to
\lim_r \pi_i A_r
\to
\lim_r \mathrm{im}(\pi_i A_r\to \pi_iB)
\to
0
\]
is exact.
Also,
\[
0\to \lim_r \mathrm{im}(\pi_i A_r\to \pi_iB)
\to
\pi_i B
\to
\lim_r \pi_{i-1}X_r
\]
is exact since $\lim_r$ is left exact.
Combine these two exact sequences to finish the proof.
\end{proof}

In view of \cite[Lemma 3.1]{Mor14},
the following result can be considered as a semistable generalization of \cite[Theorem 1.1]{Mor14}.

\begin{thm}
\label{A.5}
Let $X$ be a proper semistable fs log scheme over $(K[[t]],\langle t\rangle)$,
and let $x\in \pi_0\mathrm{KH}(\underline{X}/t)$.
Then the following conditions are equivalent:
\begin{enumerate}
\item[\textup{(i)}]
$x$ lifts to $\lim_r \pi_0 K(\underline{X}/t^r)$.
\item[\textup{(ii)}]
$\mathrm{trc}_\mathrm{dR}(x)$ lifts to $\bigoplus_{p\in \Z}\lim_r H_\setale^{2p}((X/t^r)/K,\Omega_{(X/t^r)/K}^{\geq p})$.
\end{enumerate}
\end{thm}
\begin{proof}
By \cite[(4.1)]{EM23},
we have a natural cartesian square
\[
\begin{tikzcd}
K(\underline{X}/t^r)\ar[d]\ar[r]&
\mathrm{KH}(\underline{X}/t)\ar[d]
\\
\mathrm{HC}^-(\underline{X}/t^r)\ar[r]&
L_\mathrm{cdh}\mathrm{HC}^{-}(\underline{X}/t).
\end{tikzcd}
\]
Using \eqref{A.3.1}, Proposition \ref{A.3}, and Lemma \ref{A.7},
we obtain a commutative diagram with exact rows
\[
\begin{tikzcd}
\lim_r \pi_0K(\underline{X}/t^r)\ar[d]\ar[r]&
\pi_0\mathrm{KH}(\underline{X}/t)\ar[d]\ar[r]&
\mathcal{F}\ar[d,equal]
\\
\lim_r \bigoplus_{p\in \Z} H_\et^{2p}(\underline{X}/t^r,\Omega_{(\underline{X}/t^r)/K}^{\geq p})\ar[r]&
\bigoplus_{p\in \Z}
H_\et^{2p}(\underline{X}/t,L_\mathrm{cdh}\Omega_{(\underline{X}/t)/K}^{\geq p})\ar[r]&
\mathcal{F}
\end{tikzcd}
\]
for some $\mathcal{F}$.
Since $X$ is proper and $\Omega_{(X/t^r)/K}^{\geq p}$ is a perfect complex,
$H^{2p}_\setale(X/t^r,\Omega_{(X/t^r)/K}^{\geq p})$ is finite dimensional.
Hence Proposition \ref{A.2} yields a commutative square of pro-$K$-complexes
\[
\begin{tikzcd}
\{R\Gamma_\et(\underline{X}/t^r,\Omega_{(\underline{X}/t^r)/K}^{\geq p})\}_r\ar[r]\ar[d]&
R\Gamma_\et(\underline{X}/t,L_\mathrm{cdh}\Omega_{(\underline{X}/t)/K}^{\geq p})\ar[d]
\\
\{R\Gamma_\setale(X/t^r,\Omega_{(X/t^r)/K}^{\geq p})\}_r\ar[r]&
R\Gamma_\setale(X/t,\Omega_{(X/t)/K}^{\geq p})
\end{tikzcd}
\]
satisfying the assumption of Lemma \ref{A.7}.
Whence the claim follows.
\end{proof}

\begin{rmk}
\label{A.6}
We do not know how to formulate a statement analogous to Theorem \ref{A.5} using the composite
\begin{align*}
\pi_0\mathrm{KH}(\underline{X}/t)
\xrightarrow{\mathrm{trc}_\mathrm{dR}}
&
\bigoplus_{p\in \Z}
H^{2p}_\setale((X/t)/K,\Omega_{(X/t)/K}^{\geq p})
\\
\to &
\bigoplus_{p\in \Z}
H^{2p}_\setale((X/t)/(K,\N),\Omega_{(X/t)/(K,\N)}^{\geq p}).
\end{align*}
In future work,
we plan to prove a positive characteristic analogue of Theorem \ref{A.5} generalizing \cite[Theorem 0.2]{Mor19}. 
In view of the log Tate conjecture due to Kato-Nakayama-Usui \cite[Section 1]{KNU},
working with the cohomology over $K$ instead of $(K,\N)$ would be more appropriate.
\end{rmk}

\noindent
{\bf Data Availability:} This manuscript has no associated data.\\
{\bf Declarations}\\
{\bf Conflict of interest:} The authors have no Conflict of interest.

\bibliography{bib}
\bibliographystyle{siam}

\end{document}